\documentclass[11pt,a4paper]{amsart}
\usepackage{amsfonts,amsmath,amssymb,amsthm}
\usepackage{times}
\usepackage{color}
\usepackage{subfigure}
\usepackage{graphics}
\usepackage{amstext}
\usepackage{graphicx}
\usepackage{latexsym}
\usepackage{stmaryrd}

\numberwithin{equation}{section}

\DeclareMathOperator{\const}{const}

\newcommand{\R}{\mathbb{R}}

\newcommand{\cB}{{\mathcal B}}

\newcommand{\supp}{\mathrm{supp}}

\newtheorem{theorem}{Theorem}[section]{\bf}{\it}

\newenvironment{matheorem}[1]{\begin{theorem}}{\end{theorem}}
{\bf}{\it}
{\bf}{\it}
\newtheorem{corollary}[theorem]{Corollary}{\bf}{\it}
{\it}{\rm}
\newtheorem{lemma}[theorem]{Lemma}{\bf}{\it}
\newtheorem{remark}[theorem]{Remark}{\it}{\rm}
\newtheorem{definition}[theorem]{Definition}{\bf}{\it}

\newtheorem{assumption}{Assumption}{\bf}{\it}
{\bf}{\it}
{\bf}{\it}

\title[Iterative schemes for bump solutions in a neural field model]{Iterative schemes for bump solutions in a neural field model}

\author[A.~Oleynik]{Anna Oleynik}
\address{A.~Oleynik, Department of Mathematical Sciences and Technology and
Center for Integrative Genetics, Norwegian University of Life
Sciences, N-1432 $\AA$s, Norway\\and\\
Department of Mathematics, Uppsala University, 751 06 Uppsala, Sweden}
\email{anna.oleynik@inbox.com}

\author[A.~Ponosov]{Arcady Ponosov}
\address{Department of Mathematical Sciences and Technology and
Center for Integrative Genetics, Norwegian University of Life
Sciences, N-1432 $\AA$s, Norway}
\email{arkadi.ponossov@umb.no}

\author[J.~Wyller]{John Wyller}
\address{J.~Wyller,Department of Mathematical Sciences and Technology and
Center for Integrative Genetics, Norwegian University of Life
Sciences, N-1432 $\AA$s, Norway}
\email{john.wyller@umb.no}


\begin{document}
\maketitle

\begin{abstract}
We develop two iteration schemes for construction of localized stationary solutions (bumps) of a one-population Wilson-Cowan model with a smoothed Heaviside firing rate function.
The first scheme is based on the fixed point formulation of the stationary Wilson-Cowan model. The second one is formulated in terms of the excitation width of a bump.
Using the theory of monotone operators in ordered Banach spaces we justify
convergence of both iteration schemes.
\end{abstract}

%%%%%%%%%%%%%%%%%%%%%%%%%%%%%%%%%%%%%
\section{Introduction}

Neural field models have been the subject of mathematical attention since the publications \cite{WC1,WC2,A1,A}. These models typically take the form of integro-differential equations.
We consider a one-population neural field model of the Wilson-Cowan type \cite{WC1,WC2,A1,A,Coombes2005}
\begin{equation}
\label{model}
u_t=-u+\int \limits_{-\infty}^{+\infty}\omega(y-x)f(u(y,t)-h)dy.
\end{equation}
Here $u(x,t)$ represents the activity of population, $f$ the firing-rate function, $\omega$ the connectivity function, and $h$ the firing threshold. For review on the model \eqref{model} see \cite{Coombes2005}.  Existence and stability of spatially localized solutions and traveling waves are commonly studied for the case when  the firing rate function is given by the unit step function \cite{A,Coombes2005,Coombes&Owen}. However, the results for the case when the firing rate function is smooth are few and far between \cite{KA,Coombes&Schmidt,Faugeras2008,Faugeras}.\\
\newline
%The stationary solutions of \eqref{model} satisfy the fixed point problem
%\begin{equation}
%\label{modelstationary}
%u(x)=\int\limits_{-\infty}^{+\infty}\omega(y-x)f(u(y)-h)dy.
%\end{equation}
In the mathematical neuroscience community time-independent spatially localized solutions of \eqref{model} are referred to as \emph{bumps}. The motivation for studying bumps stems from the fact that they are believed to be linked to the mechanisms of a short memory \cite{G-R1995}.  In the case when $f$ is given as a unit step function, one can find analytical expressions for the bump solutions \cite{A}. In principle, bumps solutions can also be constructed when the firing rate function is  smooth provided the Fourier-transform of the connectivity function is a real rational function. In that case the model can be converted to a higher order nonlinear differential equation which can be represented as a Hamiltonian system. The bumps are represented then by homoclinic orbits within the framework of these systems. See for example \cite{Elvin2010,LT2003,Krisner}.\\
\newline
Kishimoto and Amari \cite{KA} have proved the existence of bump solutions of \eqref{model} when $f$ is a smooth function of a special type (smoothed Heaviside function), using the Schauder fixed point theorem. The Schauder fixed point theorem, however, does not give a method for construction of the bumps.
Pinto and Ermentrout in \cite{Pinto&Ermentrout} constructed bumps using singular perturbation analysis. However, this method is quite involved,  and is restricted to the lateral-inhibitory connectivity (i.e., $\omega$ is assumed to be continuous, integrable and even, with $\omega(0) > 0$ and exactly one positive zero). Coombes and Schmidt in \cite{Coombes&Schmidt} developed an iteration scheme for constructing bumps of the model \eqref{model} with  a smoothed Heaviside function. They, however, did not give a mathematical verification of their approach. Apart from the work of Coombes and Schmidt \cite{Coombes&Schmidt}, the authors of the present paper do not know about other attempts to develop iterative algorithms for the construction of bumps.
Thus there is a need for a more rigorous analysis of iteration schemes for bumps. This serves as a motivation for the present work.\\
\newline
We present two different iteration schemes for constructing bumps. The first one is based on the fixed point problem introduced in \cite{KA}. The second scheme, which is modification of the procedure introduced in \cite{Coombes&Schmidt}, is an iteration scheme for the excitation width of the bumps. We prove that both schemes converge using the theory of  monotone operators in ordered Banach spaces.\\
\newline
The present paper is organized in the following way: In Section \ref{sec:I} the properties of the one-population Wilson-Cowan model are reviewed with emphasis on the results of Kishimoto and Amari \cite{KA}. In Section \ref{sec:Preliminaries} some necessary mathematical preliminaries are introduced. Section \ref{Sec:II} is devoted to the study of a direct iteration scheme based on the fixed point problem proposed by of Kishimoto and Amari \cite{KA}. In Section \ref{Sec:Numerics:1} we illustrate the results with a numerical example.  In Section \ref{sec:IIa} we introduce a fixed problem based on the specific representation of the firing rate function studied in \cite{Coombes&Schmidt}. The fixed problem is formulated for the crossing between bumps and a shifted parameterized threshold value $h+t,$ $t\geq 0$. The bump solution can be restored from these crossings. We prove that there is a fixed point which can be obtained by iterations. We provide an numerical example in Section \ref{Sec:Numerics:2}. In Section \ref{sec:Discussion} we summarize our findings and describe open problems.

\section{Model}
\label{sec:I}
Let $f : \R \rightarrow [0, 1]$ be an arbitrary non-decreasing function.
We assume that the connectivity function  $\omega $ satisfies the following conditions:
\begin{itemize}
\item[(i)] $\omega$ is symmetric, i.e. $\omega(-x)=\omega(x),$
\item[(ii)]$\int_{\R}|\omega(x)|dx<\infty,$ i.e., $ \omega \in L^1(\R),$
\item[(iii)]$\omega$ is continuous and bounded, i .e., $\omega \in BC(\R),$
\item[(iv)] $\omega$ is differentiable a.e. with bounded derivatives, i.e., $\omega \in W^{1,\infty}(\R).$
\end{itemize}
The examples of  such a function are
\begin{equation}
\label{eq:Mexican-hat}
\omega(x)=K e^{-kx^2}-M e^{-mx^2}, \quad 0<M<K,\; 0<m<k,
\end{equation}
and
\begin{equation}
\label{eq:w-ex2}
\omega(x)=e^{-b|x|}(b\sin|x|+\cos x), \quad b>0.
\end{equation}

In Fig.\ref{Fig1}(a) we illustrate the function given in \eqref{eq:Mexican-hat} with parameters
$K=1.5,$ $k=2,$ and $M=m=1.$  In Fig.\ref{Fig1}(b) we illustrate the function in \eqref{eq:w-ex2} with $b=0.3.$
The function \eqref{eq:Mexican-hat} models the lateral-inhibition coupling and is often called as a Mexican-hat function, e.g., see \cite{A,LT2003,Coombes2005}. The model with periodically modulated spatial connectivity given by \eqref{eq:w-ex2}  was considered in \cite{LTGE2002, Coombes&Schmidt}.
\newline
%---------FIG 1(a), FIG 1(b)----------------
\begin{figure}[h]
\centering
\subfigure[]{
\scalebox{0.5}{\includegraphics{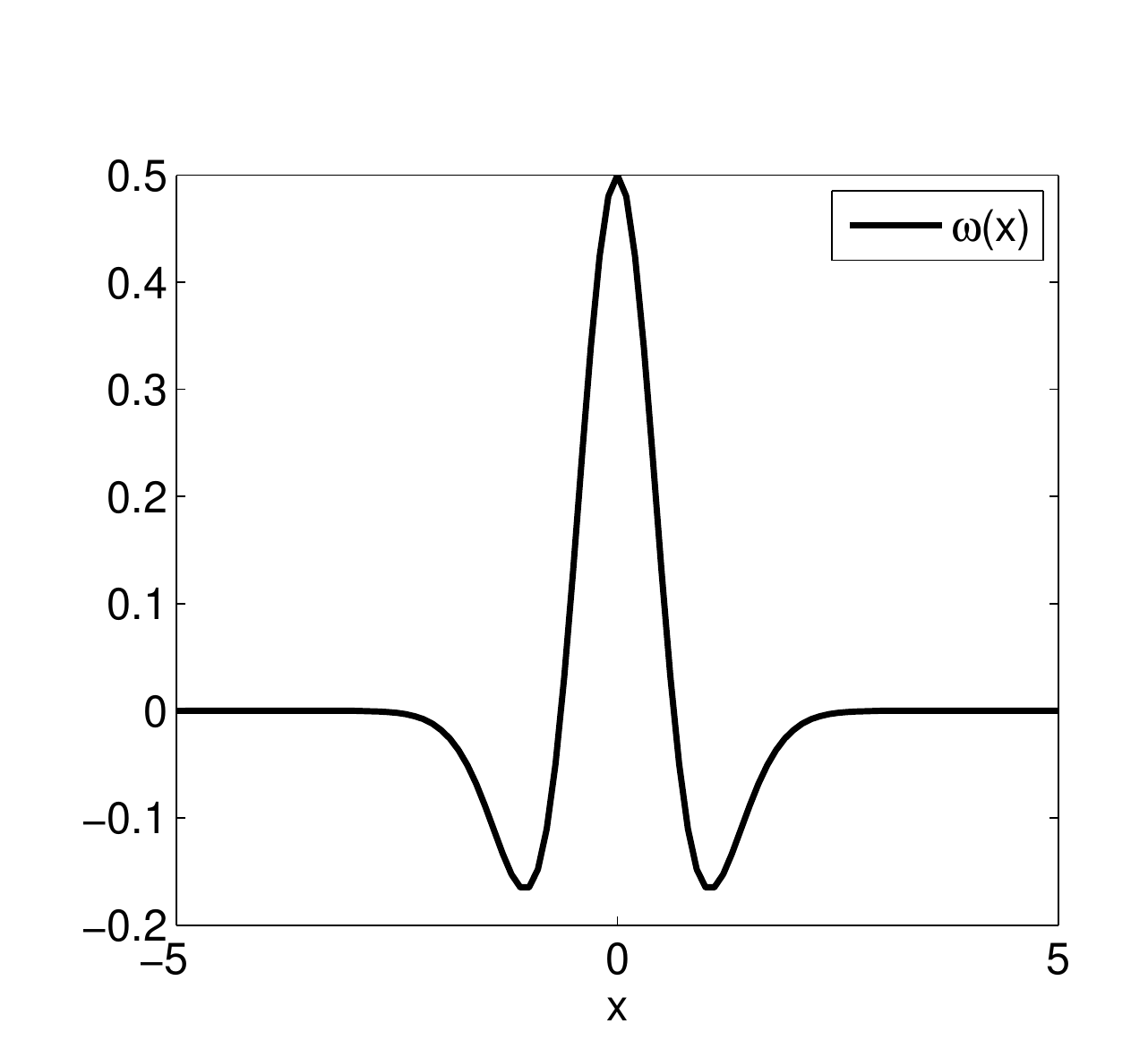}}
}
\subfigure[]{
\scalebox{0.5}{\includegraphics{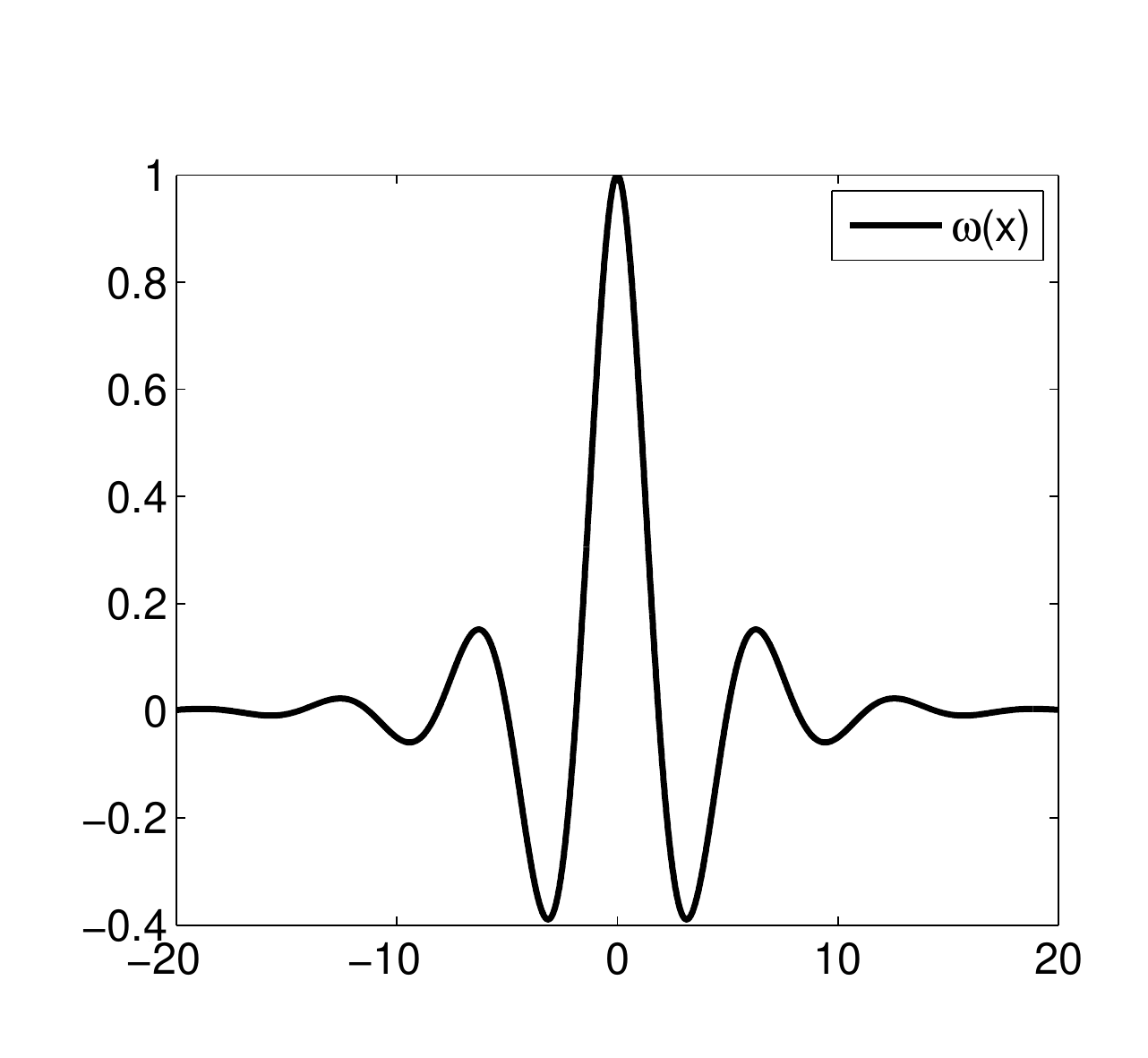}}
}
\caption{Examples of the connectivity function $\omega$: (a) The Mexican-hat function \eqref{eq:Mexican-hat}, and (b) the function \eqref{eq:w-ex2}, with the parameters given in the text.}\label{Fig1}
\end{figure}
%-----------------------------------------------------------------------------------------------------------------------------
%%%%%%%%%%%%%%%%%%%%%%%%%%%%%%%%%%%%%%%%%%%%%%%%%%%%%%%%%%%%%%%%%%%%%%

%%%%%%%%%%%%%%%%%%%%%%%%%%%%%%%%%%%%%%%%%%%%%%%%%%%%%%%%%%%%5%
Stationary solutions of \eqref{model} are given as solutions to the integral equation
\begin{equation}
\label{steady state}
u(x)=\int \limits_{-\infty}^{+\infty}w(y-x)f(u(y)-h)dy.
\end{equation}

We note the following properties of  \eqref{steady state}:
\begin{itemize}
\item{} A solution $u$ is translation invariant. That is, if $u(x)$ is a solution to \eqref{steady state}, so is $u(x-c)$ for arbitrary constant $c \in \R.$
\item{} A symmetric solution to \eqref{steady state} can be expressed as
\begin{equation}
\label{steady state:2}
u(x)=\int \limits_{0}^{+\infty}r(x,y)f(u(y)-h)dy
\end{equation}
where
\begin{equation*}
r(x,y)=w(y-x)+w(y+x).
\end{equation*}
\end{itemize}
%%%%%%%%%%%%%%%%%%%%%%%%%%%%%%%%%%%%%%%%%%%%%%%%%%%%%%%%
Let the function $f$ be given as the unit step function
\begin{equation}
\label{eq:Theta}
f=\theta, \quad \theta(u)=\left\{
\begin{array}{ll}
0,& u<0\\
1,&u\geq0.
\end{array} \right.
\end{equation}
Amari \cite{A} was the first who observed that in this case, the spatially localized solutions to \eqref{steady state} can be explicitly constructed. Following \cite{A} we introduce the following definitions:
\begin{definition}
The set $R[u]=\{x|\, \,  u(x)>h\}$ is called the excited region of $u(x),$ \cite{A}.
\end{definition}

\begin{definition}
\label{def:bump}
An equilibrium solution $u(x)$ of \eqref{model} with $f=\theta$ is called a bump with the width $a,$ if the excited region of $u$ is an interval of the length $a$, i.e., $R[u]=(a_1,a_2),$  where $a=a_2-a_1.$
\end{definition}
Then a bump solution with the width $a$ is given as
\begin{equation*}
u(x)=\int\limits_{a_1}^{a_2}\omega(y-x)dy, \quad a=a_2-a_1.
\end{equation*}
Due to translation invariance of \eqref{steady state} we without loss of generality consider bumps defined on a symmetric interval, i.e.,
\begin{equation*}
u(x)=\int \limits_{-a/2}^{a/2}\omega(y-x)dy, \quad a_{2}=-a_{1}=a/2.
\end{equation*}
It is easy to see that $u(x)$ in this form is a symmetric function. Indeed, letting $z=-y$  we have
\begin{equation*}
u(-x)=-\int \limits_{a/2}^{-a/2}\omega(-z+x)dz=\int \limits_{-a/2}^{a/2}\omega(z-x)dz=u(x).
\end{equation*}
Thus, using \eqref{steady state:2} a bump solution can be written as
\begin{equation}
u(x)=\int \limits_0^{a/2}r(x,y)dy.
\end{equation}
\newline
We define a new function $\Phi$
\begin{equation*}
\Phi(x,y)=\int_0^y r(x,z)dz, \: x,y \in \R, y>0
\end{equation*}
with
\begin{equation}
\label{eq:dPhi}
\dfrac{\partial \Phi}{\partial x}(x,y)=\omega(y+x)-\omega(y-x), \quad \dfrac{\partial \Phi}{\partial y}(x,y)=r(x,y).
\end{equation}
We conveniently express bumps by means of the function $\Phi$:
\begin{theorem}
\label{th:a-solutions}
Let $h>0$ be fixed. The  model \eqref{model} with the firing-rate function $f=\theta$ possesses a bump solution if and only if there exist a width, $a,$ such that
\begin{equation}
\label{eq:antideriv}
\Phi(a/2, a/2)\equiv\int_0^a w(y)dy=h
\end{equation}
and
\begin{itemize}
\item[(i)] $\Phi(x,a/2) \leq h ,\: \forall x> a/2,$
\item[(ii)] $\Phi(x,a/2) \geq h, \: \forall x \in [0, a/2).$
\end{itemize}
The bump solution is given then as  $u(x)=\Phi(x,a/2).$
\end{theorem}

The stability of bumps has been studied using the Amari approach \cite{A} and  the Evans function technique, \cite{Coombes2005}. Here we present the result based on \cite{A}:
\begin{theorem}
\label{th:stability}
Let $h>0$ be fixed, $f=\theta,$ and there exist a bump with the width $a.$ The bump is linearly stable if $\omega(a)<0$ and unstable if $\omega(a)>0.$
\end{theorem}

The  firing-rate function we treat here is of the following type, \cite{KA}
\begin{equation}
\label{f(u)}
f(u)=\left\{
\begin{array}{ll}
0, & u\leq 0\\
\phi(u),& 0<u<\tau \\
1,& u\geq \tau
\end{array} \right.,
\end{equation}
where $\tau>0,$ $\phi$ is an arbitrary continuous, monotonically increasing, and normalized  function such that
\begin{equation*}
\phi(0)=0, \quad  \phi(\tau)=1.
\end{equation*}
The example of  such a function is
\begin{equation}
\label{eq:logoid}
f(u)=\Sigma\left({u}/{\tau},p\right),\quad
\Sigma(u,p)=\left\{
\begin{array}{ll}
0, & u\leq 0\\
\dfrac{u^p}{u^p+(1-u)^p},& 0<u<1 \\
1,& u\geq 1
\end{array} \right., \; p>0,
\end{equation}
where $\Sigma(\cdot,p) \in C^{[p]}(\R)$  and $[p]$ denotes the integer part of $p.$
\newline

We need the following definition:
 \begin{definition}
$R^*[u]=\{x|u(x)> h+\tau\}$ is called a maximally excited region, and $R^-[u]=\{x|h<u(x)<h+\tau\}$ is an incompletely excited region, \cite{KA}.
\end{definition}

\begin{definition}
\label{def:f-bump}
An equilibrium solution $u(x)$ of \eqref{model} with $f$ given by \eqref{f(u)} is called a bump if $R^*[u]$ is the interval surrounded by an incompletely excited region $R^{-}[u],$ i.e.,
$R[u]=R^*[u]\cup R^{-}[u]$ being another interval, \cite{KA}.
 \end{definition}

Thus, by Definition \ref{def:f-bump} the function $u(x)$ displayed graphically in Fig.\ref{Fig2} can be a bump to \eqref{model} with $f$ given in \eqref{f(u)} when $\tau=\tau_1,$ whereas for $\tau=\tau_2$ it can not be a bump.

%%%%%%%%%%%%%%%%%%%%%%%%%%%%%%%%
\begin{figure}[h]
\centering
\scalebox{0.5}{\includegraphics{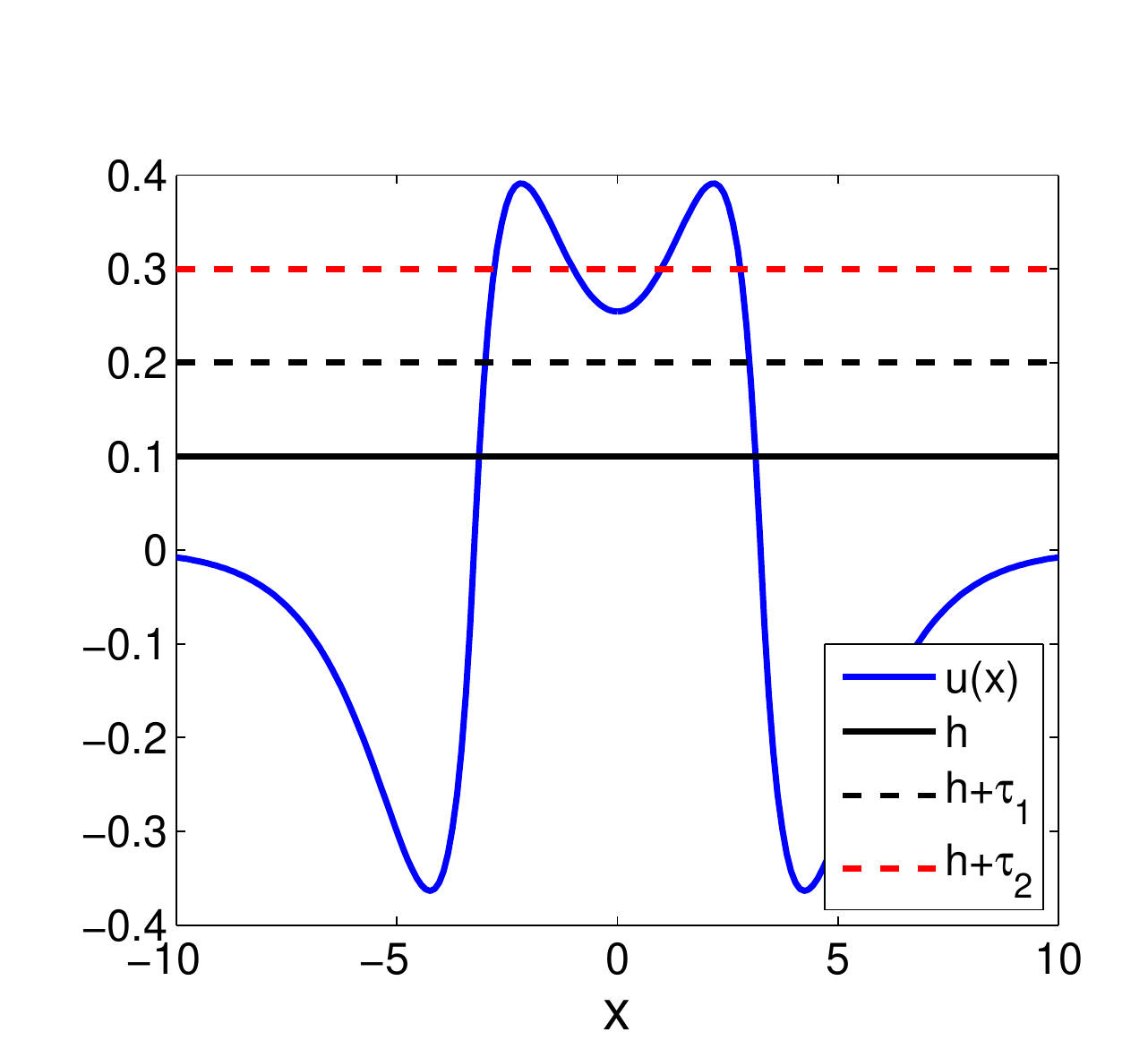}}
\caption{The graph of a function $u(x)$ which satisfies Definition \ref{def:f-bump} when $\tau=\tau_1$ and which does not satisfy it when $\tau=\tau_2$. }\label{Fig2}
\end{figure}
%%%%%%%%%%%%%%%%%%%%%%%%%%%%%%%%
Let $f$ be given as \eqref{f(u)}, $f_0(u)=\theta(u),$ and $f_\tau(u)=\theta(u-\tau).$  To distinguish between bump solutions to \eqref{model} with different firing rate functions, we use the following terminology: the neural field with the firing rate functions $f_0,$ $f_\tau,$ and  $f$  is  called a $f_0$-field,  $f_\tau$-field, and  $f$-field, respectively.
We observe that $f_\tau$-field  is equivalent to the $f_0$-field  with the new threshold  value $h+\tau,$ and
 \begin{equation*}
 \label{f0,f,fh}
 f_\tau(u) \leq f(u)\leq f_0(u).
 \end{equation*}

The original idea of Kishimoto and  Amari \cite{KA} is to use bump solutions of the $f_0$- and $f_\tau$-fields to prove the existence (and stability) of  bumps in the $f$-field.  If $\omega$ has a Mexican-hat shape (e.g., see  Fig.\ref{Fig1}(a)) then the $f_0$-field ($f_\tau$-field) possesses two symmetric bumps for moderate values of $h$, one stable and one unstable bump. In \cite{KA} it was shown, using the Schauder fixed point theorem, that there exists a bump solution of $f$-field if both $f_0$- and $f_\tau$-fields possess linearly stable bumps and $\omega$ has a Mexican-hat shape (i.e., the connectivity function can have the shape as in Fig.\ref{Fig1}(a) but not as in  Fig.\ref{Fig1}(b)).  Moreover, if  $\phi$  is a differentiable function it was  shown that the  $f$-field bump is stable.

Notice that the differentiability of $\phi$  can  be replaced by a weaker assumption, namely differentiability almost everywhere, i.e.,  $\phi \in W^{1,1}([0,\tau]).$  Then, the firing-rate function \eqref{f(u)} can be represented as in \cite{Coombes&Schmidt}, i.e.,
\begin{equation}
\label{f(u):Coombes}
f(u)=\int \limits_{-\infty}^{+\infty}\rho(\xi)\theta(u-\xi)d\xi,
\end{equation}
with $\theta$ given by \eqref{eq:Theta}, $\supp\{\rho\}=[0,\tau],$ and $\rho$  is positive and normalized $\int\limits_{-\infty}^{\infty}\rho(x)dx=1.$
\newline

In this paper we prove existence of bumps in the $f$-field, and introduce two iteration methods for their construction.
We improve the existence result obtained in \cite{KA} by relaxing on the assumption that $\omega$ has a Mexican-hat shape. We also do not require the bumps of the $f_0$- and $f_\tau$-field be stable as it is assumed in \cite{KA}.
 So far there have been two methods used to construct bumps in $f$-field: One is based on the singular perturbation analysis, \cite{Pinto&Ermentrout}. This method is quite involved and, moreover, it restricts the choice of  $\omega$ to functions of a Mexican-hat shape. The other method is to convert \eqref{steady state} to a higher order nonlinear
differential equations which can be represented as a Hamiltonian system. The bumps then
are given by homoclinic orbits within the framework of these systems, see \cite{Elvin2010,LT2003,Krisner}.
This method requires  the Fourier transform of $\omega$  to be a real rational function. Thus, it can not be applied in some cases, as for example  in the case of \eqref{eq:Mexican-hat}.
We do not requite $\omega$ to have either Mexican-hat shape or real rational Fourier transform to be able to apply our iteration schemes.

We use the following assumptions:

\begin{assumption}\label{As:1}
 There exists a bump with the width $2\Delta_0$ of the $f_0$-field model, and  a bump with the width $2\Delta_\tau$ to the $f_\tau$-field model. Moreover, the widths are such that
$\Delta_\tau<\Delta_0.$
\end{assumption}
To illustrate this assumption let us assume that there is a bump solution of the $f_0$-field model, i.e., $\Phi(\Delta_0, \Delta_0)=h,$ see Theorem \ref{th:a-solutions}. Then, by the inverse function theorem there exists a value $\tau>0$ such that $\Phi(\Delta_\tau, \Delta_\tau)=h+\tau$ for some  $\Delta_\tau<\Delta_0,$  if  $\omega(2\Delta)<0$ in some vicinity of $\Delta_0.$ In this case both bumps are stable by Theorem \ref{th:stability}. However, Assumption \ref{As:1} can be satisfied  even when the situation described above does not take place, i.e., the condition $\omega(2\Delta)<0$ is not fulfilled for all $\Delta \in [\Delta_\tau,\Delta_0],$ see for example Fig.\ref{Fig3}.

Under Assumption \ref{As:1} bumps for the $f_0$-field model and the $f_\tau$-field model are, in accordance with Theorem \ref{th:a-solutions}, given as
\begin{equation*}
\begin{array}{ll}
u_0(x)= \Phi (x, \Delta_0) & u_\tau(x)=\Phi(x,\Delta_\tau).
\end{array}
\end{equation*}

\begin{assumption}\label{As:2}
The function $r(x,y)$ is non-negative for all $x,y \in [\Delta_\tau, \Delta_0].$
\end{assumption}
We get the following relationship between $u_\tau$ and $u_0$:
\begin{lemma}
\label{lemma:u_tau<u_0}
Under Assumption \ref{As:2}  we have  $u_\tau\leq u_0$ on $[\Delta_\tau, \Delta_0].$
\end{lemma}
\begin{proof}
We get
\begin{equation*}
u_0(x)-u_\tau(x)=\int_0^{\Delta_0}r(x,y)dy -\int_0^{\Delta_\tau}r(x,y)dy=\int_{\Delta_\tau}^{\Delta_0}r(x,y)dy\geq 0.
\end{equation*}
\end{proof}

In this paper we will only consider bump solutions of the $f$-field such that
\begin{equation}
\label{eq:u-spec}
u(x)>h+\tau, \; \forall x\in R[u_\tau-\tau], \quad u(x)<h \; \forall x\not\in R[u_0].
\end{equation}

%\begin{remark}
%We will need $\rho \in L_\infty(\R)$ to prove the Frechet differentiability, see Appendix. So, for Section \ref{sec:IIa} we demand $\phi \in W^{1,\infty}(\R).$  I do not see if it is possible to avoid this restriction.
%\end{remark}

%%%%%%%%%%%%%%%%%%%%%%%%%%%%%%%%%%%%%
\section{Mathematical Preliminaries}\label{sec:Preliminaries}
%%%%%%%%%%%%%%%%%%%%%%%%%%%%%%%%%%%%%
Let $K$ be a cone in a real Banach space $\cB$ and $\leq$ be a partial ordering defined by $K.$ Let $w,v \in \cB$ be such that $w \leq v.$ Then a set of all $g\in \cB$ such that $w\leq g\leq v,$ defines an ordered interval which we denote $\llbracket w,v \rrbracket.$

The theoretical foundation of the iteration schemes presented in Section \ref{Sec:II} and Section \ref{sec:IIa} are based on the following general results:
\begin{theorem} \label{th:G&L:1}
Let $w_0, v_0 \in \cB,$ $w_0\leq v_0$ and $A: \llbracket w_0, v_0 \rrbracket \rightarrow \cB$ be an increasing operator ($Aw \leq Av$ provided $w\leq v$ for any $w,v \in \cB$) such that
\begin{equation*}
w_0 \leq Aw_0, \: Av_0 \leq v_0.
\end{equation*}
Suppose that one of the following two conditions is satisfied:
\begin{itemize}
\item[(H1)] $K$ is normal and $A$ is condensing;
\item[(H2)] $K$ is regular and $A$ is semicontinuous, i.e., $x_n\rightarrow x$ strongly implies $Ax_n \rightarrow Ax$ weakly.
\end{itemize}

Then $A$ has a maximal fixed point $x^{*}$ and a minimal fixed point $x_{*}$ in $\llbracket w_0, v_0 \rrbracket;$ moreover
\begin{equation*}
x^{*}=\lim\limits_{n \rightarrow \infty}v_n, \: x_{*}=\lim\limits_{n \rightarrow \infty}w_n,
\end{equation*}
where $v_n=Av_{n-1}$ and $w_n=Aw_{n-1},$ $n=1,2,3...,$ and
$$ w_0 \leq w_1 \leq ...\leq w_n \leq ...\leq v_n \leq...\leq v_1 \leq v_0.$$
See \cite{Guo&La}.
\end{theorem}

From Theorem \ref{th:G&L:1} we get the following corollary.
\begin{corollary}
\label{corollary:1}
If under the conditions of Theorem \ref{th:G&L:1} $x^{*}=x_{*}=\tilde{x},$ then $\tilde{x}$ is the unique fixed point of the operator $A$ in $\llbracket w_0,v_0\rrbracket.$
\end{corollary}

\begin{theorem}
\label{th:cone}
The cone $ K=\{ u \in \cB| u(x) \geq 0 \}$ is normal  but not regular in $\cB=C(\bar{D}),$ and regular in $\cB=L_p(D), \; 1 \leq p< \infty,$
where $D$ is a  bounded set and $\bar{D}$ is a closed bounded set, see \cite{Guo&La}.
\end{theorem}

%\begin{theorem}
%\label{th:A1A2compact}
%The superposition operator $A=A_1 A_2$ is compact if either
%\begin{itemize}
%\item[1.] $A_2$ is compact and $A_1$ is continuous;  or
%\item[2.] $A_1$ is compact and $A_2$ is bounded.
%\end{itemize}
%\end{theorem}
%
%\begin{proof}
%By definition of the compact operator, i.e.,
%The operator $A$ from a Banach space $X$ to a Banach space $Y$ is compact if it maps any bounded subset of $X$ to relatively compact subset of $Y.$ Under the first set of conditions we have that any image of the bounded set from $X$ is a relatively compact set. Since that any continues operator has the property to map a relatively compact set to a relatively compact set \textbf{ref?}, continuity of $A_1$ complete the proof.
%The proof is even more straightforward under second set of conditions.
%\end{proof}

\begin{theorem}
\label{th:Hammerstein}
The Hammerstein operator
\begin{equation*}
(Af)(x)=\int_a^b k(x,y) \psi(y,f(y))dy
\end{equation*}
is continuous and compact in $C([a,b])$ if $k(x,y)$ and  $\psi(x,y)$ are continuous functions on $[a,b]\times[a,b]$.
\end{theorem}

\begin{proof}
The operator $A$ can be represented as the superposition, $A=LN,$ where $L$ is the linear operator
\begin{equation*}
(Lg)(x)=\int_a^b k(x,y)g(y)dy ,
\end{equation*}
and $N$ is the Nemytskii operator
\begin{equation*}
(Nf)(x)=\psi(x,f(x)).
\end{equation*}

The linear operator $L: C([a,b]) \rightarrow \R$  is continuous and compact  if $k(x,y)$ is continuous \cite{Kolmogorov}.
Obviously, the  Nemytskii operator  $N: C([a,b]) \rightarrow C([a,b])$ is continuous and bounded if $\psi(x,y)$ is continuous.
Thus, the Hammerstein operator $A$ is completely continuous as the superposition of the continuous and bounded operator $N,$ and completely continuous operator $L.$
\end{proof}

%%%%%%%%%%%%%%%%%%%%%%%%%%%%%%%%%%%%%%%%%%%%%%%%%%
\section{Iteration Scheme I: Direct Iteration.} \label{Sec:II}
%%%%%%%%%%%%%%%%%%%%%%%%%%%%%%%%%%%%%%%%%%%%%%%%%
In this section we consider the direct iteration scheme for construction of bumps. This scheme is based on \cite{KA}. We start out by observing that a bump solution of an $f$-field satisfying \eqref{eq:u-spec}
can be written  as
\begin{equation}
\label{eq:restriction of u}
u(x)=u_\tau(x)+\int_{\Delta_\tau}^{\Delta_0} r(x,y)f(u(y)-h)dy,
\end{equation}
for all $x\in [\Delta_\tau, \Delta_0].$\\

First, we prove that there exists a solution $u^*(x)$ to \eqref{eq:restriction of u} and it can be iteratively constructed. Next, we introduce assumptions under which $u^*(x)$ is appeared to be a restriction of a bump solution to \eqref{model} on $[\Delta_\tau,\Delta_0].$ Finally, in Section we illustrate our results numerically and draw some conclusions based on the numerical observation.\\

Let $\cB$ be a real Banach space with partial ordering $\geq$ defined by the cone $K=\{u\in \cB| u(x)\geq 0\}.$ We have the following theorem.
\begin{theorem}\label{th:existence}
Let $\cB$ be either $L_2([\Delta_\tau, \Delta_0])$  or $C([\Delta_\tau, \Delta_0]).$
Let $\omega$  satisfy Assumption \ref{As:1} and \ref{As:2}, the operator $T_f: \llbracket u_\tau, u_0\rrbracket\subset \cB \rightarrow \cB$ be defined as
\begin{equation}
\label{operator:T_f}
(T_f u)(x)=u_\tau(x)+\int_{\Delta_\tau}^{\Delta_0}r(x,y)f(u(y)-h)dy.
\end{equation}
Then $T_f$ has a fixed point in $\llbracket u_\tau, u_0 \rrbracket.$ Moreover, the sequences $\{T_f^n u_\tau\}$ and $\{T_f^n u_0\}$ converge to the minimal and maximal fixed point of the operator $T_f,$ respectively.
\end{theorem}
\begin{proof}
We base our proof on Theorem \ref{th:G&L:1}.
The cone $K$ is normal provided  $\cB=C([\Delta_\tau, \Delta_0])$ and is regular provided $\cB=L_2([\Delta_\tau, \Delta_0]),$  see Theorem \ref{th:cone}. By Assumptions \ref{As:1} and \ref{As:2}  there exist $u_\tau$ and $u_0$ such that $0\leq u_\tau(x)\leq u_0(x)$ for all $x\in [\Delta_\tau, \Delta_0].$
Thus, $\llbracket u_\tau, u_0\rrbracket$ is the ordered interval defined on $\cB.$\\

 We describe the properties of $T_f$ which hold true in both spaces: $\cB=C([\Delta_\tau, \Delta_0])$ and $\cB=L_2([\Delta_\tau, \Delta_0]).$
First of all, $T_f$ is  positive and  monotone due to Assumption \ref{As:2} and monotonicity of $f,$ i.e.,
\begin{equation*}
u_1(x)\leq u_2(x)\Rightarrow(T_fu_1)(x) \leq (T_fu_2)(x).
\end{equation*}
Moreover, $T_f$ is continuous due to continuity of  $f$ and boundedness of $r(x,y).$
Defining a non-linear operator $T_g$ associated with the non-negative function $g(x)$ by
\begin{equation*}
(T_gu)(x)=u_\tau(x)+\int_{\Delta_\tau}^{\Delta_0} r(x,y)g(u(y)-h)dy
\end{equation*}
we get
\begin{equation*}
\begin{array}{ll}
T_{f_0}u_0=u_0,& T_{f_\tau}u_\tau=u_\tau.
\end{array}
\end{equation*}
Thus, from Assumption \ref{As:2}  we can easily deduce that
\begin{equation*}
g(x)\leq m(x) \, \Rightarrow \, (T_gu)(x)\leq (T_mu)(x),
\end{equation*}
and, therefore,
\begin{equation*}
(T_{f_\tau}u)(x)\leq (T_fu)(x) \leq (T_{f_0}u)(x).
\end{equation*}
We obtain
\begin{equation*}
\begin{array}{ll}
T_fu_\tau \geq T_{f_\tau}u_\tau =u_\tau,& T_fu_0 \leq T_{f_0}u_0=u_0.
\end{array}
\end{equation*}

From Theorem \ref{th:G&L:1}  we conclude that $T_f:\llbracket u_\tau, u_0 \rrbracket \subset L_2([\Delta_\tau, \Delta_0]) \rightarrow L_2([\Delta_\tau, \Delta_0])$ has a fixed point in  $\llbracket u_\tau, u_0 \rrbracket$ which can be found by iterations.
However, for the case $\cB=C([\Delta_\tau, \Delta_0])$ it remains to show
that $T_f$ is condensing.
Applying  Theorem \ref{th:Hammerstein} to the Hammerstein operator on the right hand side of \eqref{operator:T_f}, i.e., to the operator $u \mapsto \int_{\Delta_\tau}^{\Delta_0}r(x,y)f(u(y)-h)dy,$
 we find that $T_f: \llbracket u_\tau, u_0 \rrbracket  \rightarrow C([\Delta_\tau, \Delta_0])$  is compact and, thus, condensing. This observation completes the proof.
\end{proof}

Next we show that the fixed point $u^*$ of the operator $T_f$ referred to in the theorem above can be extended to the solution $u$ of \eqref{steady state} over  $\R$ in such a way that $u(x)\geq h+\tau$ for $x\in [0,\Delta_\tau]$ and $u(x)\leq h$ for $x\in [\Delta_0, \infty).$ To do so we introduce additional assumptions on the connectivity function $\omega$.

\begin{assumption}\label{As:3}
$u_0$ is a decreasing function on the interval $[\Delta_\tau, \Delta_0]$ which is equivalent to
$$\dfrac{\partial \Phi}{\partial x}(x,\Delta_0)<0 ,\: \forall x\in [\Delta_\tau, \Delta_0],$$
and
$u_\tau$ is a  decreasing function on $[\Delta_\tau, \Delta_0]$ which is equivalent to
$$\dfrac{\partial \Phi}{\partial x}(x,\Delta_\tau)<0 ,\: \forall x\in [\Delta_\tau, \Delta_0].$$
\end{assumption}
From this assumption the transversality of the intersections $u_0(x)$ with $h,$ and $u_\tau(x)$ with $h+\tau$ follows. Thus, the assumption always can be satisfied if, for example, we choose a small $\tau$ provided $|\Delta_0-\Delta_\tau|$ is sufficiently small.
%
%\begin{assumption}\label{As:4}
%\begin{equation*}
% \int_{\Delta_\tau}^{\Delta_0} \mid\dfrac{\partial^2 \Phi}{\partial y\partial x }(x,y)\mid dy  < -\dfrac{\partial \Phi}{\partial x}(x,\Delta_\tau), \: \forall x\in [\Delta_\tau, \Delta_0].
%\end{equation*}
%\end{assumption}

%%%%%%%%%%%%%%%%%%%%%%%%%%%%%%%%%%%%%%%%%%%%%%%%%
\begin{assumption}\label{As:4}
\begin{equation*}
 \int_{\Delta_\tau}^{\Delta_0} \mid\dfrac{\partial r}{\partial x}(x,y)\mid f(u_0(y)-h)dy < -\dfrac{\partial \Phi}{\partial x}(x,\Delta_\tau) ,\: \forall x\in [\Delta_\tau, \Delta_0].
\end{equation*}
\end{assumption}
%%%%%%%%%%%%%%%%%%%%%%%%%%%%%%%%%%%%%%%%%%%%%%%%%%%
Assumption \ref{As:4} is technical and is used to prove that $u^*(x)$ is a decreasing function on $[\Delta_\tau,\Delta_0].$
Noticing  that
\begin{equation*}
\dfrac{\partial r}{\partial x}=\dfrac{\partial^2 \Phi}{\partial x \partial y },
\end{equation*}
 non-negativity of  ${\partial r}/{\partial x}$ for  $x,y \in [\Delta_\tau, \Delta_0]$ and Assumption 3 imply that Assumption 4 is satisfied.  Indeed, the following chain of inequalities is valid for all $x\in [\Delta_\tau,\Delta_0]$
 \begin{equation*}
 \begin{split}
 \int_{\Delta_\tau}^{\Delta_0} \mid\dfrac{\partial r}{\partial x}(x,y)\mid f(u_0(y)-h)dy + \dfrac{\partial \Phi}{ \partial x} (x,\Delta_\tau)<\\
  <\int_{\Delta_\tau}^{\Delta_0} \dfrac{\partial^2 \Phi}{\partial y \partial x}(x,y)dy+ \dfrac{\partial \Phi}{ \partial x} (x,\Delta_\tau)=  \dfrac{\partial \Phi}{ \partial x} (x,\Delta_0)<0.
 \end{split}
\end{equation*}
However, the non-negativity of ${\partial r}/{\partial x}$ is rather rigid condition.

\begin{lemma} \label{lemma:u'<0}
Let Assumption 3 and 4 be fulfilled. Then, the fixed point $u^*(x)$ of the operator $T_f$ is differentiable and  decreasing on the interval $[\Delta_\tau, \Delta_0].$
\end{lemma}
\begin{proof}
We get
\begin{equation*}
(u^*(x))'=u'_\tau(x)+I, \quad I=\int_{\Delta_\tau}^{\Delta_0} \dfrac{\partial r}{\partial x}(x,y) f(u^*(y)-h)dy.
\end{equation*}
In order to prove that $(u^*(x))'<0$ we need to show that $I< -u'_{\tau}(x) $ where
\begin{equation*}
u'_{\tau}(x)=\dfrac{\partial \Phi}{\partial x}(x,\Delta_\tau) <0
\end{equation*}
by Assumption \ref{As:3}.

We get the following chain of inequalities for $|I|$
\begin{equation*}
\begin{split}
|I| &\leq \int_{\Delta_\tau}^{\Delta_0}   \mid\dfrac{\partial r}{\partial x}(x,y)\mid f(u^*(y)-h)dy \leq \\
& \leq \int_{\Delta_\tau}^{\Delta_0}   \mid\dfrac{\partial r}{\partial x}(x,y)\mid f(u_0(y)-h)dy .
\end{split}
\end{equation*}
Thus, by Assumption \ref{As:4} we have $|I|< -u'_\tau(x)$ and therefore $I< -u'_\tau(x).$
\end{proof}

%Instead of Assumption \ref{As:4} we can introduce the following assumption:
%\begin{matheorem}{Assumption 4 $'$}
%\label{As:4'}
%\begin{equation*}
% \int_{\Delta_\tau}^{\Delta_0} \mid\dfrac{\partial r}{\partial x}(x,y)\mid f(u_0(y)-h)dy < -\dfrac{\partial \Phi}{\partial x}(x,\Delta_\tau) ,\: \forall x\in [\Delta_\tau, \Delta_0].
%\end{equation*}
%\end{matheorem}
%Although this assumption is less restrictive than Assumption \ref{As:4}, it implicitly contains a restriction on the firing rate function, $f$.

Finally, we introduce the assumption which by Definition \ref{def:f-bump} allows us to view the extended solution $u$ of $u^*$ as a bump:

\begin{assumption} \label{As:5}
The function $\Phi$ is such that
\begin{itemize}
\item[(i)] $\Phi(x,y) \leq h, \: \forall x> \Delta_0, \; y \in [\Delta_\tau, \Delta_0],$\\
\item[(ii)] $\Phi(x,y) \geq h+\tau, \: \forall x \in [ 0,\Delta_\tau],\; y \in [\Delta_\tau, \Delta_0].$
%\item[(iii)] $\Phi(x,\Delta_\tau) \leq h+\tau, \mbox{ and }  \Phi(x,\Delta_0) \geq h, \: \forall x \in [ \Delta_\tau, \Delta_0].$
\end{itemize}
\end{assumption}

\begin{theorem}
\label{th:extension}
Let $u^* \in C^1([\Delta_\tau, \Delta_0])$ define the fixed point of the operator $T_f$ referred to in Theorem \ref{th:existence}. Under Assumptions \ref{As:3} - \ref{As:5} the function $u^*$ can be extended to a bump solution $u(x)$ of \eqref{steady state:2} defined on $\R$ in such a way that $u(x)>h+\tau$ for all $x \in [0, \Delta_\tau) $ and $u(x)<h$ for all $x \in (\Delta_0, \infty).$
\end{theorem}
\begin{proof}
From Theorem \ref{th:existence} and Lemma \ref{lemma:u'<0} it follows that there exist unique $\delta_\tau,$ $\delta_0:$ $\Delta_\tau<\delta_\tau<\delta_0<\Delta_0$ such that
\begin{equation*}
u^*(\delta_\tau)=h+\tau, \quad u^*(\delta_0)=h.
\end{equation*}
Let us introduce the function $F$ defined by
\begin{equation*}
F(y)=\left\{\begin{array}{ll}
            1,&  0\leq y < \delta_\tau \\
            f(u^*(y)-h), &\delta_\tau \leq y \leq \delta_0\\
            0, & y>\delta_0.
            \end{array}\right.
\end{equation*}
Then, according to Lemma \ref{lemma:u'<0} $F(y)$ is monotonically decreasing function  on $[\delta_\tau, \delta_0]$ with $F(\delta_0)=0$ and $F(\delta_\tau)=1.$
From \eqref{steady state} we get
\begin{equation*}
u(x)=\int_{0}^{\delta_0} r(x,y)F(y)dy=-\int_{0}^{\delta_0} r(x,y) \int_y^{\delta_0} F'(z)dz.
\end{equation*}
By changing the order of integration, we have
\begin{equation*}
u(x)=-\int_{\delta_\tau}^{\delta_0} \int_0^z r(x,y)dy F'(z)dz= \int_{\delta_\tau}^{\delta_0} \int_0^z r(x,y)dy d(1-F(z)),
\end{equation*}
or
\begin{equation*}
u(x)=\int_{0}^{1} \int_0^{z(\xi)} r(x,y)dy d\xi=\int_{0}^{1} \Phi(x,z(\xi)) d\xi, \; \mbox{ where } \xi=1-F(z).
\end{equation*}
It remains to show that $\Phi(x,z)< h$ for $ x> \Delta_0,$ $z \in [\delta_\tau, \delta_0],$ and $\Phi(x,z)> h+\tau$ for $ 0\leq x< \Delta_\tau,$ $z \in [\delta_\tau, \delta_0].$ Assumption \ref{As:5} guarantees that these inequalities are fulfilled even on a larger interval. Moreover, $u(-x)=u(x)$ due to symmetry of $\omega.$ Thus, the proof is completed.
\end{proof}
The proof of Theorem \ref{th:existence} is a modification of the theorem used in \cite{KA}. The modification is caused by the fact that our assumptions on the connectivity functions $\omega$ are different from ones used in \cite{KA}.

\subsection{Numerical example}\label{Sec:Numerics:1}
In this section we exploit examples of $\omega$ be given by \eqref{eq:Mexican-hat}.
Thus, the equation  $\omega(x)=0$ has one positive solution $x_{max}=\sqrt{1/{(k-m)}\ln(K/M) }.$ Furthermore, $\omega(x)$ is positive for all $|x|<x_{max}$ and is negative for all $|x|>x_{max}.$ This defines the behavior of the antiderivative to $\omega,$ i.e.,
 \begin{equation}
 \label{eq:W}
 W(x)=\int\limits_0^x \omega(y)dy.
 \end{equation}
Then, for any $h,\tau$ satisfying the inequality
\begin{equation*}
\lim \limits_{x\rightarrow +\infty}W(x)<h<h+\tau<W(x_{max})
\end{equation*}
 there exist the widths $\Delta_\tau, \Delta_0$ referred to in Assumption 1 such that $ [\Delta_\tau, \Delta_0] \subset (1/2\, x_{max},\infty.)$ Moreover, $u_0(x)=\Phi(x,\Delta_0)$ and  $u_{\tau}=\Phi(x,\Delta_\tau)$ are bumps  of the $f_0$- and the $f_\tau$-field model.\\

 We let $K=1.5,$ $k=2,$ $M=m=1$(see Fig.\ref{Fig1}), and chose $h=0.1$ and $\tau=0.05.$ Given these parameters and \eqref{eq:antideriv} we found $\Delta_0=0.6633$ and two possible values of $\Delta_\tau:$ $0.1769$ and $0.5012$ (see Fig.\ref{Fig3}(a)).  We fix $\Delta_\tau=0.1769,$ and denote $\Delta_\tau^{st}=0.5012.$  We notice here that $\omega(2\Delta_0)$ and $\omega(2\Delta_\tau^{st})$ are negative while $\omega(2\Delta_\tau)$ is positive, see Fig. \ref{Fig3}(b). By Theorem \ref{th:stability}, we conclude that $u_0(x)=\Phi(x,\Delta_0)$ and $u_\tau^{st}=\Phi(x,\Delta_\tau^{st})$ are linearly stable solution to $f_0$- and $f_\tau$-field models, while $u_\tau(x)=\Phi(x,\Delta_\tau)$ is a linearly unstable solution to $f_\tau$-field model. This explains upper index '$st$' in our notation.

%%%%%%%%%%%%%%%%%%%%%%%%%%%%%%%%%%%%%%%%%%%%%%%%%%%%%%%%%%%%%%%%%
\begin{figure}[h]
\centering
\subfigure[]{
\scalebox{0.5}{\includegraphics{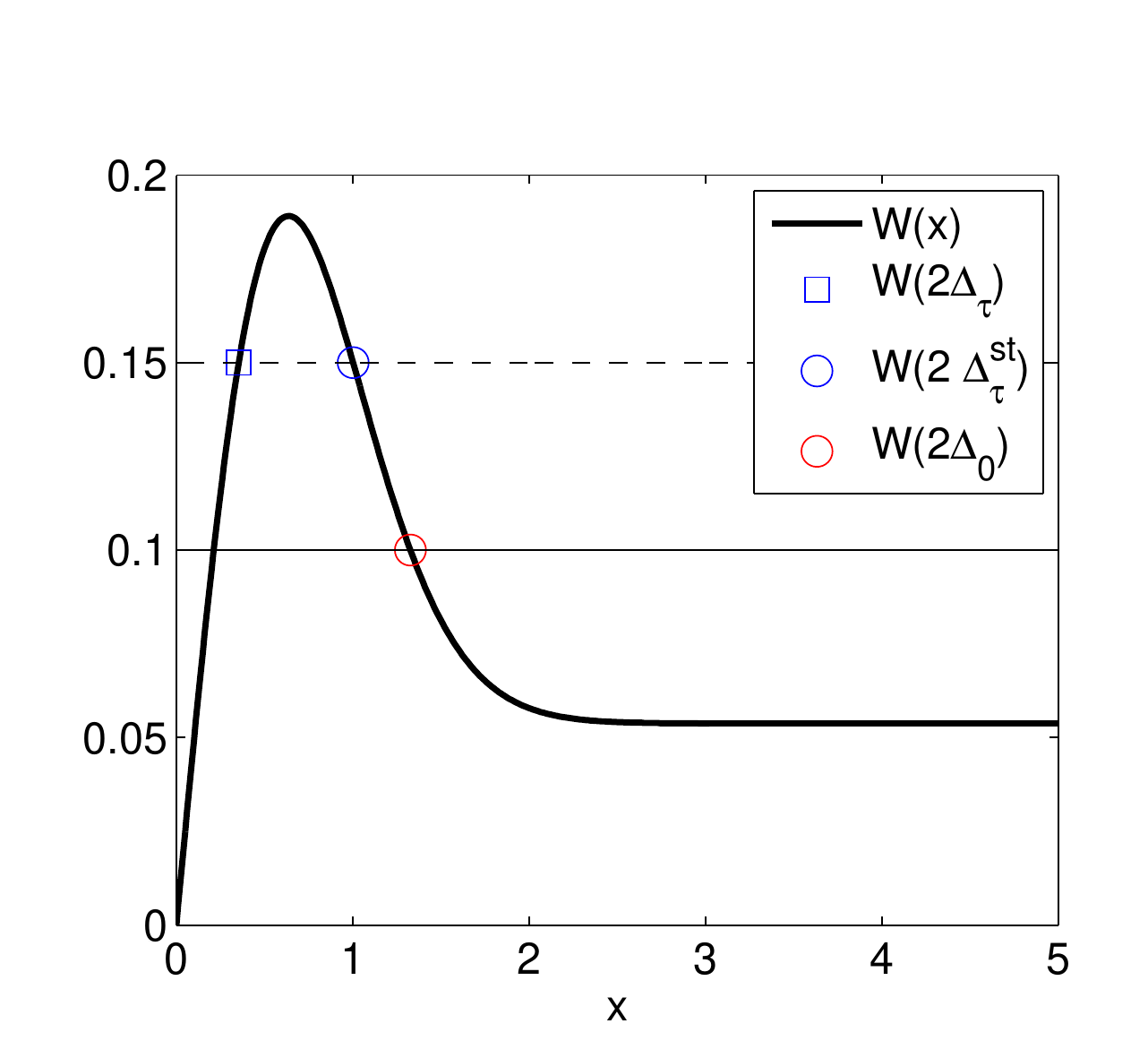}}
}
\subfigure[]{
\scalebox{0.5}{\includegraphics{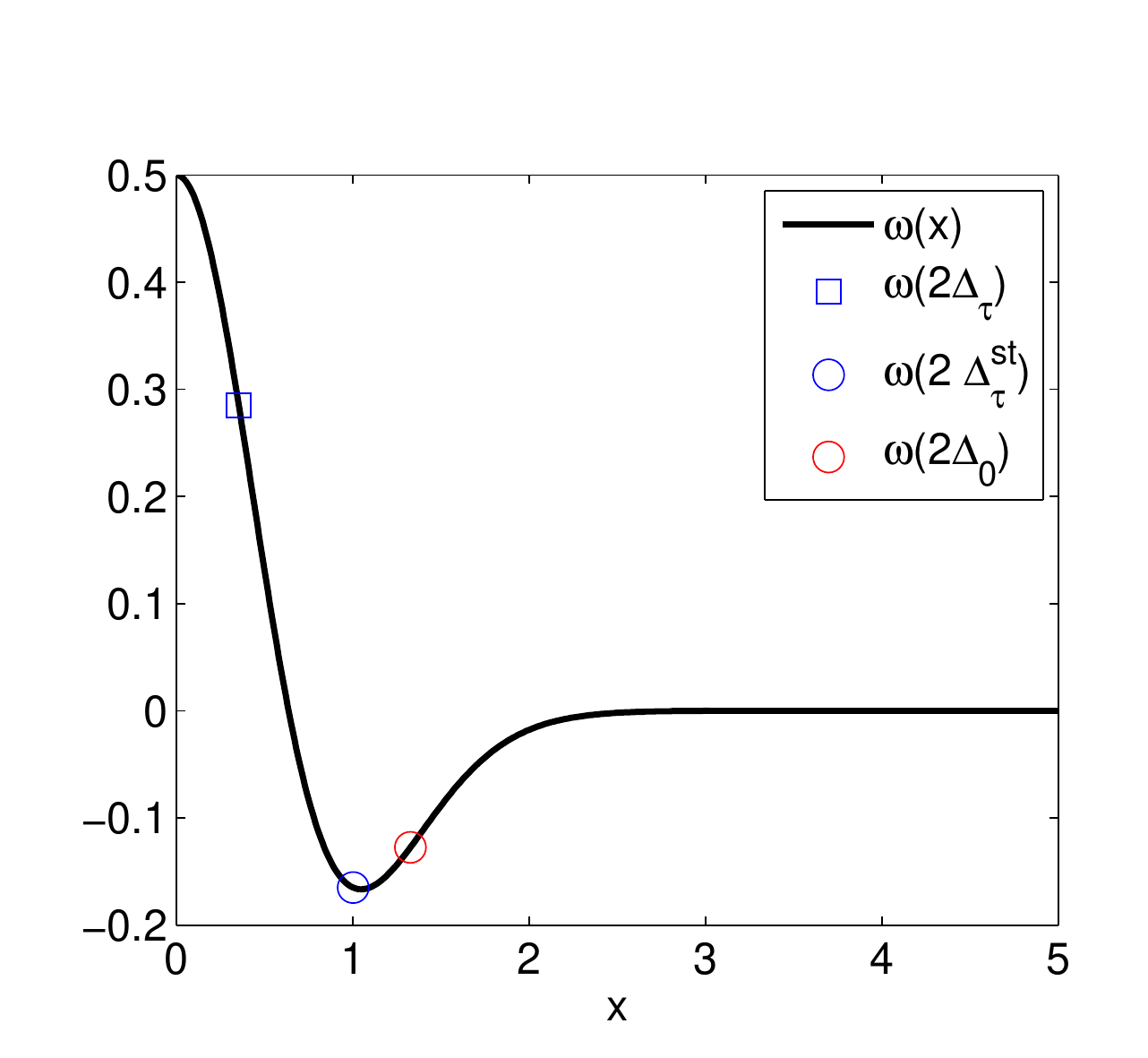}}
}
\caption{(a) The antiderivative of $\omega$ given by \eqref{eq:W}.  The blue square indicates the point $(2\Delta_\tau,h+\tau),$ the blue circle corresponds to $(2\Delta_\tau^{st},h+\tau),$ and the red circle corresponds to $(\Delta_0^{st},h).$ The horizontal solid and dashed line are $h$ and $h+\tau,$ respectively. (b) The connectivity function $\omega$ and the points $(a, \omega(a))$ with $a=\Delta_\tau,$ $\Delta_\tau^{st}$ and $\Delta_0,$ denoted by the blue square, blue circle, and  red circle, respectively.} \label{Fig3}
\end{figure}
%%%%%%%%%%%%%%%%%%%%%%%%%%%%%%%%%%%%%%%%%%%%%%%%%%%%%%%%%%%%%%%%%%%%%

Assumptions 2 - 5 has been verified numerically. Thus, we apply Theorem \ref{th:existence}. We chose $f$ to be given as in
in \eqref{eq:logoid} with $p=3.$ In Fig.\ref{Fig4}(a) we have plotted the fixed point $u^*(x)$ of the operator $T_f$ obtained by iterations from the restriction of $u_\tau(x)$ and $u_0(x)$ on $[\Delta_\tau,\Delta_0]$ as the initial values.  From Corollary \ref{corollary:1} we conclude that $u^*$ is, then, a unique solution of the fixed point problem for \eqref{operator:T_f}. We have plotted $u_0,$ $u_\tau,$ and $u^{st}_\tau$ on the same figure to illustrate the inequality
\begin{equation}
\label{eq:u<u<u}
u^{st}_\tau\leq u^*\leq u_0.
\end{equation}
Thus, $u^*$ is located in between of two stable bumps of the $f_0-$ and the $f_\tau$ field model on $[\Delta_\tau,\Delta_0]$. Based on \cite{A} we claim  that $u^*$ is a restriction of the stable bump solution to $f-$field equation.

Fig.\ref{Fig4}(b) illustrates the dynamics of the iteration process. There we have plotted the numerical errors calculated as
\begin{equation}
\label{eq:errors}
\varepsilon(n)=\max\limits_{x}|(T_f^{n}u_0)(x)-(T_f^{n}u_\tau)(x)|, \quad n=1,2,...,N,
\end{equation}
where $T_f^{0}$ is equivalent to the identity operator, $n$ corresponds to the iteration number, and  $N$ denotes the total number of iterations. We observe that in our calculations $\varepsilon(n)<10^{-5}$ for $n\geq 16.$
%%%%%%%%%%%%%%%%%%%%%%%%%%%%%%%%
\begin{figure}[h]
\centering
\subfigure[]{
\scalebox{0.5}{\includegraphics{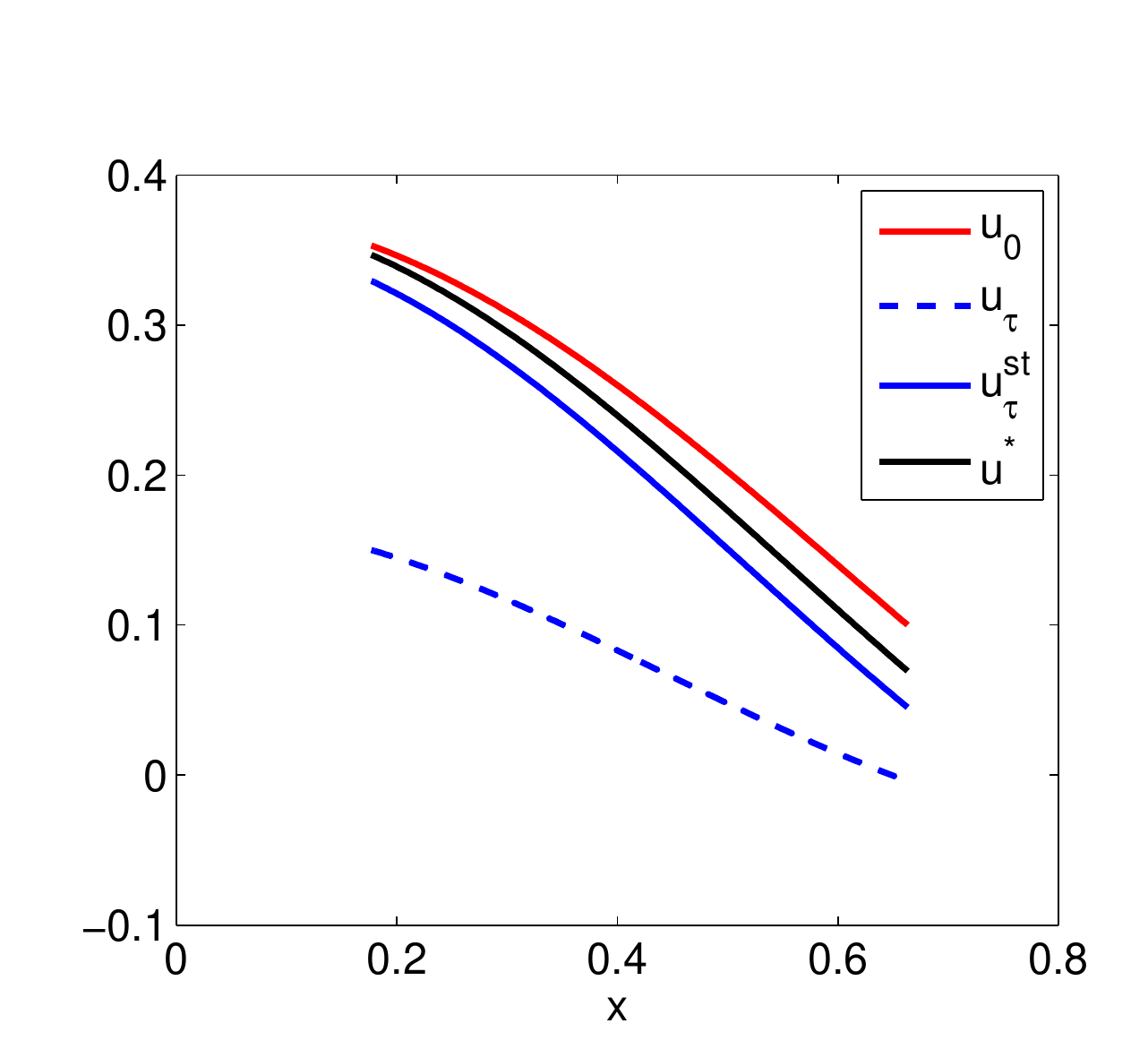}}
}
\subfigure[]{
\scalebox{0.5}{\includegraphics{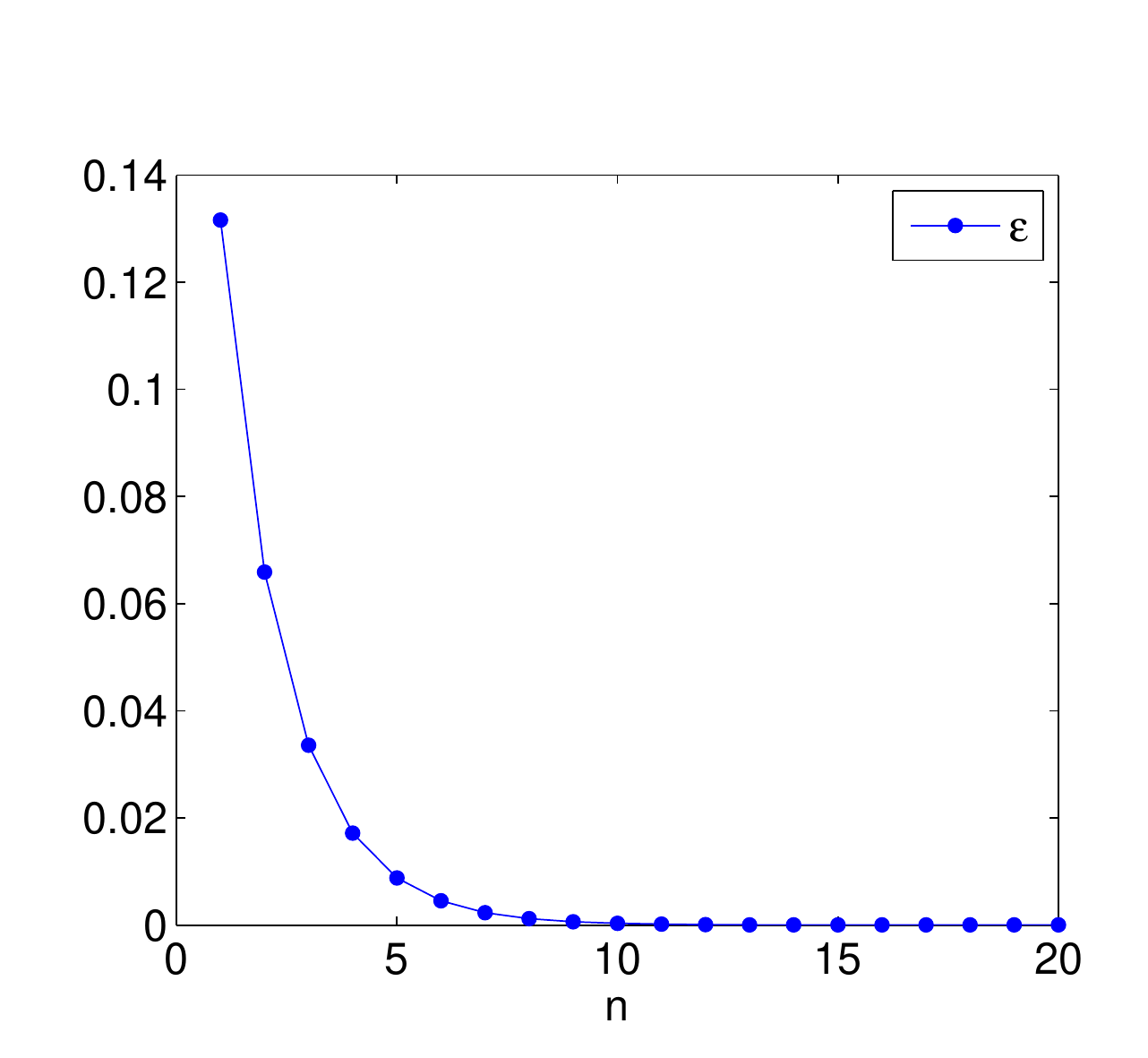}}
}
\caption{(a) A fixed point of the operator \eqref{operator:T_f}, $u^*,$ with the bump of $f_0$-field model, $u_0,$ and the bumps of the $f_\tau$-field model, $u^{st}_\tau$ and $u_\tau,$ given on $[\Delta_\tau, \Delta_0].$ The connectivity $\omega$ is as in \eqref{eq:Mexican-hat} and $f$ as in \eqref{eq:logoid}. The parameters are given as in the text. (b) The error sequence defined as in \eqref{eq:errors}. }\label{Fig4}
\end{figure}
%%%%%%%%%%%%%%%%%%%%%%%%%%%%%%%%

In Fig.\ref{Fig5} we have plotted the
graph of the bump solution to the $f$-field model obtained as the extension of $u^*$ from $[\Delta_\tau,\Delta_0]$ to $\R,$ see Theorem \ref{th:extension}

%%%%%%%%%%%%%%%%%%%%%%%%%%%%%%%%
\begin{figure}[h]
\centering
\scalebox{0.5}{\includegraphics{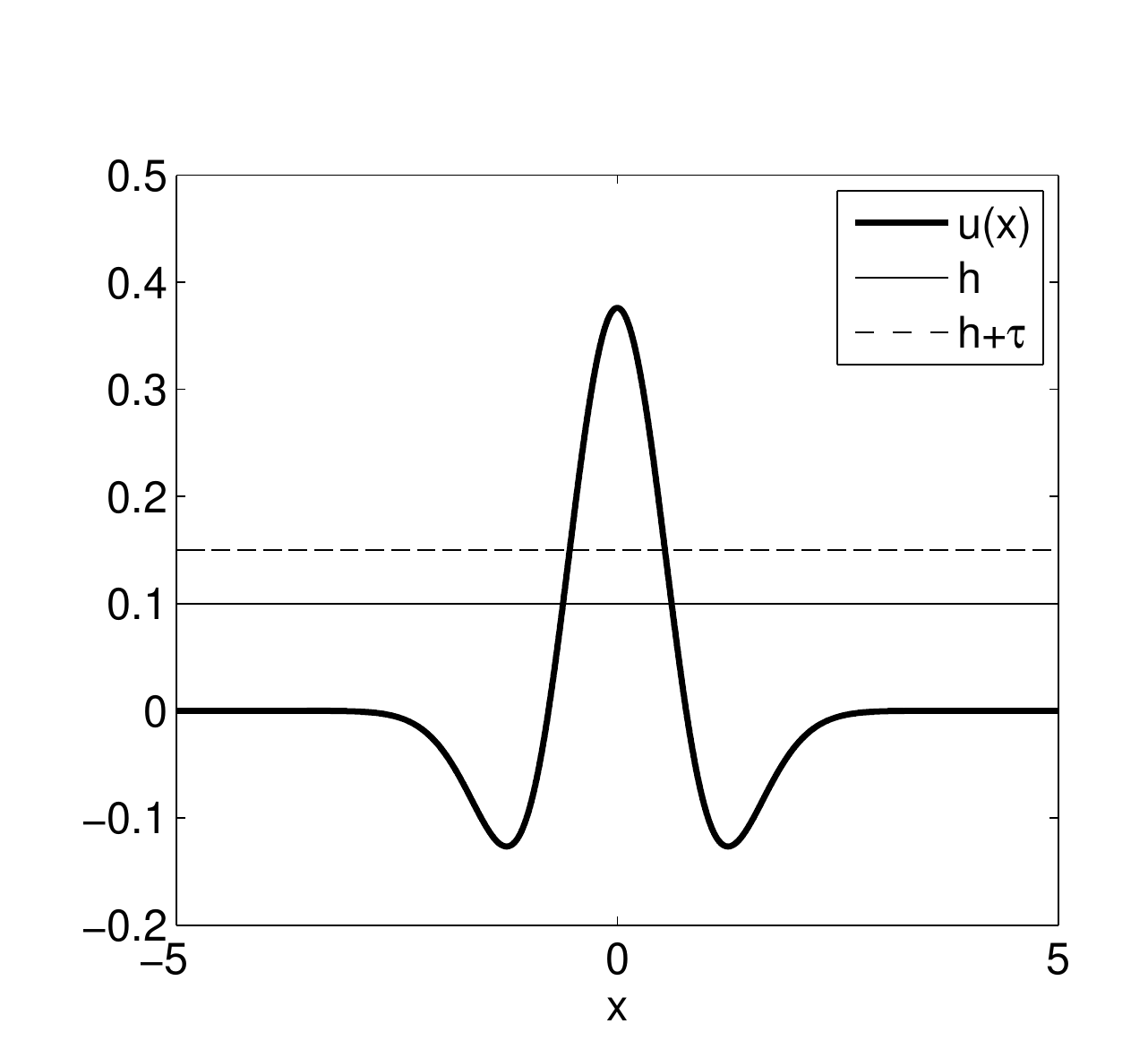}}
\caption{ The bump solution of the $f$-field constructed from $u^*,$ see Fig.\ref{Fig4}(a). The parameters used are as in the text. }\label{Fig5}
\end{figure}
%%%%%%%%%%%%%%%%%%%%%%%%%%%%%%%%

Assumptions 1-5 are satisfied for some set of parameters when $\omega$ is given as in \eqref{eq:w-ex2}. We, however, do not provide this example here.
%%%%%%%%%%%%%%%%%%%%%%%%%%%%%%%%%%%%%%%%%%%%%%%%%%%%%%%%%%%%%%%%%
\section{Iteration Scheme II: Bumps Width Iteration} \label{sec:IIa}
%%%%%%%%%%%%%%%%%%%%%%%%%%%%%%%%%%%%%%%%%%%%%%%%%%%%%%%%%%%%%%%%%
In \cite{Coombes&Schmidt} an iteration procedure for construction of bump solutions to the $f$-field model has been worked out. However, a mathematical verification of the procedure  has not been given. In the present section we introduce an iteration scheme which is based on the idea presented in \cite{Coombes&Schmidt} and give a mathematical verification of this approach.
we illustrate the results with a numerical example using the same parameters as in Section \ref{Sec:Numerics:1}.

For $t \in [0,\tau]$ we assume that  there exist  the excitation width $\Delta(t)$ such that a bump solution to $f$-field model, $u_\Delta (x),$ satisfies
\begin{equation*}
u_\Delta(\pm \Delta(t))=t+h.
\end{equation*}
Then $u_{\Delta}(x)$ can be described by
\begin{equation}
\label{uDelta}
u_{\Delta}(x)=\int_0^\tau  \rho(\xi) \int_{0}^{\Delta(\xi)}r(x,y)dy d\xi
\end{equation}
using the representation \eqref{f(u):Coombes}.

Let $\cB$ be a real Banach space with partial ordering $\geq$ defined by a cone $E=\{u\in \cB| u(x)\geq 0\}.$ In this section we assume $\cB=L_2([0,\tau]).$ Then, if a bump of the $f$-field is given by \eqref{eq:u-spec} then $\Delta(t) \in \llbracket \Delta_\tau, \Delta_0 \rrbracket.$
The excitation width $\Delta$ satisfies the fixed point problem
\begin{equation}
\label{x=Ax}
\Delta=A \Delta, \quad (A\Delta)(t) = \Delta(t) + k\left(u_\Delta(\Delta(t))-t-h\right), \quad k=\const \in \R.
\end{equation}

\begin{theorem}
\label{th:Frechet}
The operator $A$ is Fr\'{e}chet differentiable in $L_2[0,\tau]$ if $f \in W^{1,\infty}(\R)$.
\end{theorem}

\begin{proof}
Let us define the operator
\begin{equation*}
(G\Delta)(t)=u_{\Delta}(\Delta(t)).
\end{equation*}
We calculate the Fr\'{e}chet derivative of the operator $G\Delta$. To do this we first compute its G\'{a}teaux derivative
\begin{equation*}
\label{Gateaux_def}
g(\delta)=\lim_{\lambda \rightarrow 0}\dfrac{G(\Delta+\lambda \delta)-G(\Delta)}{\lambda}.
\end{equation*}
For any $\delta\in L_2[0,\tau]$ let us consider
\begin{equation}
\label{eq:Gautex:nominator}
G(\Delta+\lambda \delta)-G(\Delta)=\int_0^\tau \rho(\xi)I_1(\xi)d\xi,
\end{equation}
where
\begin{equation*}
I_1(\xi)=\int_{0}^{\Delta(\xi)+\lambda \delta(\xi)}r(\Delta(t)+\lambda \delta(t),y)dy -
\int_{0}^{\Delta(\xi)}r(\Delta(t),y)dy.
\end{equation*}
Making use of the Taylor expansion for $r(\Delta(t)+\lambda \delta(t),y)$ as a function of $\lambda$ at $\lambda=0$ we have
\begin{equation*}
\begin{split}
I_1(\xi)&=\int_{0}^{\Delta(\xi)+\lambda \delta(\xi)}\left(r(\Delta(t),y)+ \lambda \dfrac{\partial r}{\partial x}(\Delta(t),y) \delta(t)+ o(\lambda)\right)dy-\int_{0}^{\Delta(\xi)}r(\Delta(t),y)dy\\
&= \int_{\Delta(\xi)}^{\Delta(\xi)+\delta(\xi)}r(\Delta(t),y)dy+ \lambda \delta(t)\int_0^{\Delta(\xi)+\lambda \delta(\xi)} \dfrac{\partial r}{\partial x}(y,\Delta(t)) dy+ o(\lambda).
\end{split}
\end{equation*}
Plugging $I_1$ into (\ref{eq:Gautex:nominator}) and making use of the mean value theorem we get the following formula
\begin{equation}
\label{eq:g(delta)}
g(\Delta,\delta(t))= \int_0^\tau \rho(\xi)\delta(\xi) r(\Delta(t),\Delta(\xi))d\xi+\delta(t) \int_0^\tau \rho(\xi)\dfrac{\partial \Phi}{\partial x}(\Delta(t),\Delta(\xi))d\xi.
\end{equation}

Hence, we arrive at the conclusion that the G\'{a}teaux derivative is a linear operator. In order to prove Fr\'{e}chet differentiability of the operator $G$ we show, in accordance with \cite{Zeidler}, that $g(\cdot,\delta): L_2[0,\tau]\rightarrow L_2[0,\tau]$ is a continuous operator for all $\delta \in L_2[0,\tau]$. The proof of this fact is technical and we therefore formulate it as a separate lemma:

\begin{lemma}
The operator $g(\cdot,\delta): L_2[0,\tau]\rightarrow L_2[0,\tau]$ is  continuous  for all $\delta \in L_2[0,\tau].$
\end{lemma}
\begin{proof}

We consider the first and the second integral of \eqref{eq:g(delta)} separately as the operators of $\Delta.$  Using  the Cauchy-Schwarz  and Minkowskii inequalities we show that these operators are continuous and, thus, $g(\cdot,\delta)$ is continuous as well, for any $\delta \in L_2[0,\tau].$
We present the proof for the first integral operator. The proof of continuity for the second term proceeds in the same way and is omitted.

We introduce
\begin{equation*}
(F\Delta)(t)=\int_0^\tau \rho(\xi)\delta(\xi) r(\Delta(t),\Delta(\xi))d\xi.
\end{equation*}
We obtain
\begin{equation*}
\begin{split}
(F\Delta_1- F\Delta_2)(t)=\int_0^\tau \rho(\xi)\delta(\xi)& \left( r(\Delta_1(t),\Delta_1(\xi))- r(\Delta_1(t),\Delta_2(\xi))+\right.\\
&\left.+ r(\Delta_1(t),\Delta_2(\xi))- r(\Delta_2(t),\Delta_2(\xi))\right)d\xi=I_1(t)+I_2(t)
\end{split}
\end{equation*}
where by the mean value theorem $I_1$ and $I_2$ can be defined as
\begin{equation*}
I_1(t)=\int_0^\tau \rho(\xi)\delta(\xi) \dfrac{\partial r}{\partial y}(\Delta_1(t),\tilde{\Delta}_1(\xi))(\Delta_1(\xi)-\Delta_2(\xi))d\xi
\end{equation*}
\begin{equation*}
I_2(t)=\int_0^\tau \rho(\xi)\delta(\xi)\dfrac{\partial r}{\partial x}(\tilde{\Delta}_2(t),\Delta_2(\xi))(\Delta_1(t)-\Delta_2(t))d\xi
\end{equation*}
with $\tilde{\Delta}_k=\lambda_k \Delta_1+(1-\lambda_k) \Delta_2,$ for some $\lambda_k \in [0,1],$ $k=1,2.$

We consider the norm of the difference. Using the Minkowskii inequality we get
\begin{equation*}
\begin{split}
||F\Delta_1-F\Delta_2||_{L_2[0,\tau]}=\left(\int_0^\tau(I_1(t)+I_2(t))^2dt\right)^{1/2} \leq  \\
\leq \left(\int_0^\tau|I_1(t)|^2 dt \right)^{1/2}+ \left(\int_0^\tau|I_2(t)|^2 dt\right)^{1/2}
\end{split}
\end{equation*}
Applying the Cauchy - Schwarz inequality to each of the terms we have
\begin{equation*}
\begin{array}{l}
||F\Delta_1-F\Delta_2||_{L_2[0,\tau]}\leq\\
\leq \left(\int \limits_0^{\tau} \int\limits_0^\tau  \left| \rho(\xi) \delta(\xi) \dfrac{\partial r}{\partial y}
(\Delta_1(t),\tilde{\Delta}_1(\xi))\right|^2 d\xi \, \int\limits_0^\tau |\Delta_1(\xi)-\Delta_2(\xi)|^2d\xi \,dt\right)^{1/2}+\\
+ \left(\int\limits_0^{\tau} \int\limits_0^\tau  \left| \rho(\xi) \delta(\xi) \dfrac{\partial r}{\partial y}(\Delta_1(t),\tilde{\Delta}_1(\xi))\right|^2 d\xi \, \int\limits_0^\tau |\Delta_1(t)-\Delta_2(t)|^2 d\xi \,dt\right)^{1/2}
\end{array}
\end{equation*}
Since $r\in W^{1,\infty}(\R \times \R),$  $\rho \in L_\infty(\R)$, and $\delta \in L_2[0,\tau]$ the following estimate is valid
\begin{equation*}
\begin{array}{l}
\int\limits_0^\tau  \left| \rho(\xi) \delta(\xi) \dfrac{\partial r}{\partial y}(\Delta_1(t),\tilde{\Delta}_1(\xi))\right|^2 d\xi  \leq C^2/2\\
\int\limits_0^\tau  \left| \rho(\xi) \delta(\xi) \dfrac{\partial r}{\partial x}(\tilde{\Delta}_2(t),{\Delta}_1(\xi))\right|^2 d\xi \leq C^2/2
\end{array}
\end{equation*}
where
$$C^2=2||\rho||_{L_{\infty}(\R)} \max \{||\dfrac{\partial r}{\partial x}||_{L_{\infty}(\R \times \R)}, ||\dfrac{\partial r}{\partial y}||_{L_\infty(\R \times \R)}\} ||\delta||_{L_2[0,\tau]}^2.$$
Therefore, we get the following inequality
\begin{equation*}
||F\Delta_1-F\Delta_2||_{L_2[0,\tau]}\leq |C|\sqrt{\tau} ||\Delta_1-\Delta_2||_{L_2[0,\tau]}
\end{equation*}
from which the continuity of $F$ follows.
\end{proof}
For convenience we make redefinition:  $g(\Delta,\delta)=G_{\Delta}'\delta.$
Obviously, the operator $A$ is Fr\'{e}chet differentiable in any $\Delta \in L_2[0,\tau]$ and
\begin{equation}
\label{eq:A'}
A_\Delta'=I+kG_{\Delta}'
\end{equation}
\end{proof}
We have the following lemma:
\begin{lemma}
\label{lemma:A'>0}
The operator $A_{\Delta}'\delta \geq 0$ for $\delta \geq 0$ and $\Delta \in \llbracket \Delta_\tau, \Delta_0 \rrbracket$ under Assumption \ref{As:2} and \ref{As:3} and $0<k<1/m,$ where
\begin{equation}
\label{eq:m}
m=-\min \limits_{t,\xi \in [0,\tau]} \dfrac{\partial \Phi}{\partial x}(\Delta(t),\Delta(\xi)).
\end{equation}
\end{lemma}
\begin{proof}
First of all, we notice that $\partial \Phi/\partial x \in BC(\R) \times BC(\R).$ Thus, there exists a finite minimum of $\partial \Phi/\partial x$ on the given set. Moreover, this minimum is negative according to Assumption \ref{As:3}.
Therefore, $m$ given by \eqref{eq:m} is finite and positive, and  the operator $A'_\Delta$ preserves  positivity for   $0<k<1/m.$
\end{proof}

\begin{theorem}
\label{th:existence2}
Let the conditions of Theorem \ref{th:Frechet} and Lemma \ref{lemma:A'>0} be satisfied. Then the operator $A: \llbracket\Delta_\tau, \Delta_0 \rrbracket \rightarrow D \subset L_2[0,\tau]$ has a fixed point  in $ \llbracket \Delta_\tau, \Delta_0 \rrbracket.$ Moreover, the sequences $\{A^n \Delta_\tau\}$ and $\{A^n \Delta_0\}$ converge to the smallest and greatest fixed point of the operator $A,$ respectively.
\end{theorem}
\begin{proof}

The operator $A$ is monotonically increasing. Indeed, we let $\Delta_2 \geq \Delta_1.$ Then $A \Delta_2 -A\Delta_1=A_{\Delta}'(\Delta_2-\Delta_1)$ where $\Delta \in \llbracket\Delta_1, \Delta_2 \rrbracket \subset \llbracket\Delta_\tau, \Delta_0 \rrbracket.$ Using Lemma \ref{lemma:A'>0} we conclude that $A$ is monotone.

 The operator $A$ is Fr\'{e}chet differentiable, and hence continuous in $L_2[0,\tau]$ (see Lemma \ref{th:Frechet}). Moreover, we have the following inequalities
\begin{equation*}
(A \Delta_0)(t)=\Delta_0+k(u_{\Delta_0}(\Delta_0)-t-h)= \Delta_0 + k(u_0(\Delta_0)-t-h)=\Delta_0-kt\leq \Delta_0.
\end{equation*}
and
\begin{equation*}
(A \Delta_\tau)(t)=\Delta_\tau+k(u_{\Delta_\tau}(\Delta_\tau)-t-h)= \Delta_\tau + k(u_\tau(\Delta_\tau)-t-h)=\Delta_0+k(\tau-t)\geq \Delta_\tau.
\end{equation*}
Applying Theorem \ref{th:G&L:1}we complete the proof.
\end{proof}

\begin{remark}
\label{remark:1}
We prove Theorem \ref{th:existence2} for the case when $D\in L_2[0,\tau]$ but do not consider the case $D\in C[0,\tau].$ The cone of positive functions in $C[0,\tau]$ is not regular. Therefore additional assumptions on the operator $A$ are required (see Theorem \ref{th:G&L:1}). We notice that $A$ is not compact in $C[0,\tau].$ Indeed, the operator $A$ is a Fr\'{e}chet differentiable with $A_\Delta'$ defined as in \eqref{eq:A'} where $A_\Delta'$ is a sum of the identity operator and a compact operator, thus is not compact. Therefore, $A$ is not a compact operator, see \cite{Zeidler}. The operator $A$ does not seem to be condensing either, at least with respect to the Hausdorff measure. The case of more general measures of noncompactness \cite{Guo&La,AKPS1982} is not considered here.
\end{remark}

It remains to show that $u_\Delta,$ where $\Delta$ is the fixed point of \eqref{x=Ax}, is a bump.
The definition of  $u_\Delta$ requires $\Delta(t)$ monotonically decreasing. We introduce the assumption.
\begin{matheorem}{Assumption 3$'$}
\label{As:3'}
The partial derivative of $\Phi$ with respect to $x$ is negative for  $x=\Delta(t),$ $y=\Delta(s)$ for $t,s \in [0,\tau]$ and $\Delta(t)$ is a fixed point of \eqref{x=Ax}, i.e.,
\begin{equation*}
\dfrac{\partial \Phi}{\partial x}(\Delta(t),\Delta(s))<0 ,\: \forall t,s\in [0, \tau].
\end{equation*}
\end{matheorem}

\begin{lemma}
The fixed point $\Delta(t)$ of operator $A$ is monotonically decreasing and differentiable on $[0,\tau]$  under Assumption 3$'$.
\end{lemma}
\begin{proof}
Since $\Delta(t)$ is a solution of the fixed point problem \eqref{x=Ax} then $u_\Delta(\Delta(t))=t+h.$
We prove the lemma by direct differentiation of the last equality with respect to $t.$ We obtain
\begin{equation*}
\int_0^\tau \rho(\xi) \dfrac{\partial \Phi}{\partial x}(\Delta(t),\Delta(\xi)) \Delta'(t) d\xi=1.
\end{equation*}
Thus,
\begin{equation*}
\Delta'(t)=\left(\int_0^\tau \rho(\xi) \dfrac{\partial \Phi}{\partial x}(\Delta(t),\Delta(\xi)) d\xi \right)^{-1} <0
\end{equation*}
as $\dfrac{\partial \Phi }{\partial x} (\Delta(t),\Delta(\xi))<0$ by Assumption 3$'$.
\end{proof}

 Assumption 3$'$  requires an apriori knowledge of $\Delta(t)$ and therefore can not be checked before $\Delta(t)$ is found. Thus, we suggest to replace this assumption with the following one:

\begin{matheorem}{Assumption 3$''$}
\label{As:3''}
The partial derivative of $\Phi$ with respect to $x$ is negative for all  $x,y\in [\Delta_\tau, \Delta_0],$ i.e.,
\begin{equation*}
\dfrac{\partial \Phi}{\partial x}(x,y)<0 ,\: \forall x,y\in [\Delta_\tau, \Delta_0].
\end{equation*}
\end{matheorem}
The fulfillment of Assumption 3$''$ implies that Assumption 3 and Assumption 3$'$ are satisfied.

In addition to Assumption 3$'$ (or 3$''$) we have the following requirement:

\begin{matheorem}{Assumption 5$'$}
\label{As:5'}
The function $\Phi$ is such that
\begin{itemize}
\item[(i)] $\Phi(x,y) \leq h, \: \forall x> \Delta(0), \; y \in [\Delta(\tau), \Delta(0)],$\\
\item[(ii)] $\Phi(x,y) \geq h+\tau, \: \forall x \in [ 0,\Delta(\tau)],\; y \in [\Delta(\tau), \Delta(0)].$
\end{itemize}
\end{matheorem}

\begin{theorem}\label{th:extention2}
Let $\Delta$ be a fixed point refereed to in Theorem \ref{th:existence2}. Then  $u_\Delta$ defined as \eqref{uDelta} is a bump solution to \eqref{model} under Assumptions 3$''$  and 5$'$.
\end{theorem}
\begin{proof}
We rewrite \eqref{uDelta} as
\begin{equation*}
u_\Delta(x)=\int_0^{\tau}\rho(\xi)\Phi(x,\Delta(\xi)) d\xi.
\end{equation*}
Next, we make use of Assumption 5$'$. Keeping in mind the normalization property of $\rho$ we show that
\begin{equation*}
\begin{array}{ll}
u_\Delta(x)\leq h, & \forall x> \Delta(0), \; y \in [\Delta(\tau), \Delta(0)],\\
\\
u_\Delta(x)\geq h+\tau, &\forall x \in [ 0,\Delta(\tau)],\; y \in [\Delta(\tau), \Delta(0)].
\end{array}
\end{equation*}
\end{proof}

\begin{remark}
For operator $T_f$ we use Assumptions \ref{As:1}-\ref{As:5}, and   Assumptions \ref{As:1}-\ref{As:2}, 3$''$ and 5$'$ for the operator $A.$ Assumptions 3$''$ and 5$'$ are more restrictive than Assumptions \ref{As:3} and \ref{As:5}. Moreover Assumption \ref{As:5} needs information about the fixed point $\Delta(t)$ which is a disadvantage. On the other hand, the operator $T_f$ requires one extra assumption, Assumption \ref{As:4}.
\end{remark}

\subsection{Numerical example}\label{Sec:Numerics:2}
Let $\omega(x),$ $h,$ $\tau$ and $\Delta_\tau,$ $\Delta_\tau^{st},$ $\Delta_0$ are chosen as in Section \ref{Sec:Numerics:1}.
Them, as we have mentioned before, Assumptions \ref{As:1},\ref{As:2}, and \ref{As:3} hold true. Hence, we can apply Theorem \ref{th:existence2} and obtain $\Delta(t).$ In Fig.\ref{Fig6}(a) we illustrate the result of the iteration process.  In Fig.\ref{Fig6}(b) we have plotted the errors calculated as
\begin{equation}
\label{eq:errors:2}
\varepsilon(n)=\max\limits_{x}|(A^{n}|\Delta_0)(t)-(A^{n}\Delta_\tau)(t)|, \quad n=1,2,...,N.
\end{equation}
Similar as in \eqref{eq:errors}, $A^{0}$ defines the identity operator, $n$ corresponds to the iteration number, and  $N$ denotes the total number of iterations.
In our calculations $\varepsilon(n)<10^{-5}$ for $n\geq 13,$ and the minimal and maximal fixed points converges to each other.  Thus, the fixed point is unique, see Corollary \ref{corollary:1}.
We also observe that the fixed point $\Delta(t)$ belongs to $\llbracket \Delta^{st}_\tau, \Delta_0 \rrbracket,$ see Fig.\ref{Fig6}(a).
%%%%%%%%%%%%%%%%%%%%%%%%%%%%%%%%
\begin{figure}[h]
\centering
\subfigure[]{
\scalebox{0.5}{\includegraphics{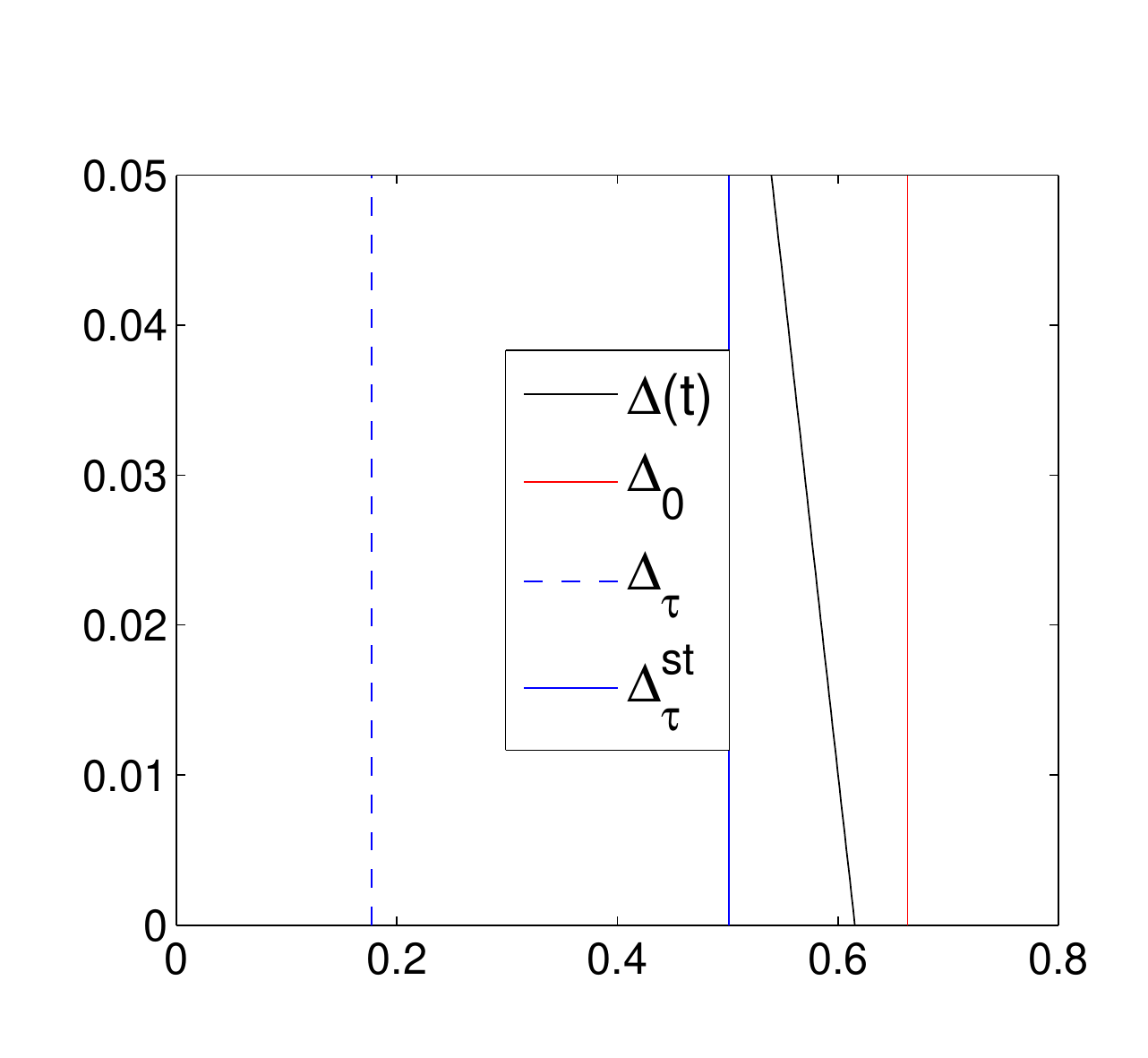}}
}
\subfigure[]{
\scalebox{0.5}{\includegraphics{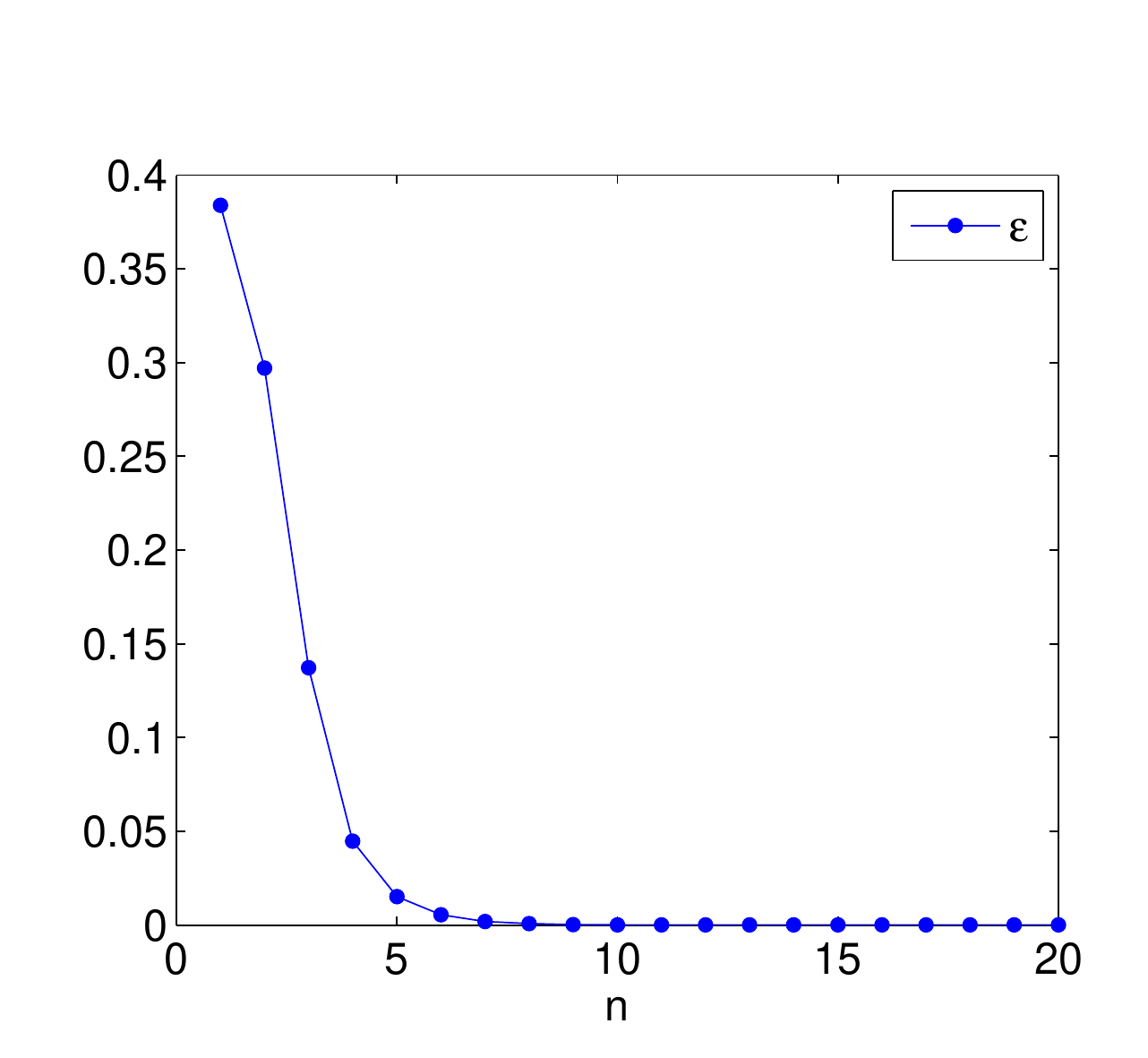}}
}
\caption{
(a) A fixed point of \eqref{x=Ax}, $\Delta(t),$  vertical lines $\Delta_0,$ $\Delta_\tau,$ and$\Delta^{st}_\tau$ as they defined in Section \ref{Sec:Numerics:1}.  The connectivity function $\omega$ is given as in Fig.\ref{Fig1}(a), $f$ is defined by \eqref{eq:logoid} with $p=3,$ $h=0.1,$ $\tau=0.05.$ (b) The errors given as in \eqref{eq:errors:2}.
}\label{Fig6}
\end{figure}
%%%%%%%%%%%%%%%%%%%%%%%%%%%%%%%%

Knowing $\Delta(t)$ we have checked that Assumption 5$'$ is fulfilled. Thus, by Theorem \ref{th:extention2} we can obtain a bump solution to the $f$-field model \eqref{model}.

We claim that that this bump coincides with the bump constructed in Section \ref{Sec:Numerics:1}. To demonstrate this, we found $\delta(t)$ that solves
\begin{equation*}
u^*(\delta(t))=t+h, \quad t\in [0,\tau]
\end{equation*}
with $u^*$ being  the fixed point of the operator $T_f,$ see Fig.\ref{Fig4}(a).
We calculate the relative error as
\begin{equation}
\epsilon=\max\limits_t \left|\frac{\Delta(t)-\delta(t)}{\delta(t)}\right|.
\end{equation}
For our example we have obtained $\epsilon=2.5 \times 10^{-3}.$ We notice here that our implementation is not optimal and can be significantly improved. We do not pursue this problem here, however.
%%%%%%%%%%%%%%%%%%%%%%%
\section{Discussion}
\label{sec:Discussion}

We have introduced two iteration schemes for finding a bump  solution in the $f-$field of the Wilson-Cowan model: The first scheme is based on the fixed point problem formulated by Kishimoto and Amari \cite{KA}. The second one is described by the fixed point problem formulated for the interface dynamics of the bump. The latter formulation became possible due to the special representation of the firing rated function introduced by Coombes and Schmidt \cite{Coombes&Schmidt}.

We have proved using the theory of  monotone operators in Banach spaces that both iteration schemes converge under Assumption \ref{As:1} and \ref{As:2}.
From the iterative procedures we obtain the solution on the finite interval $[\Delta_\tau, \Delta_0]$ (see Section \ref{Sec:II}), and on $[\Delta(0),\Delta(\tau)]$  (see Section \ref{sec:IIa}).
Then it has been  shown that under some additional assumptions on the connectivity function $\omega$ this solution determines a bump of the $f$-field on $\R.$

The assumptions imposed for the first method (see Section \ref{Sec:II}) differ from the ones imposed for the second method (see Section \ref{sec:IIa}). The evident disadvantage of Assumption 3$'$ and 5$'$ is that they contain  information about  the output of the iteration procedure, $\Delta(t)$. Assumption 3$'$ can be substituted with the more restrictive Assumption 3$''$, but not Assumption 5$'$. Thus, Assumption 5$'$ can not be checked in advance. On the other hand, the set of assumptions for the fixed point problem \eqref{x=Ax} is in general less restrictive than the assumptions imposed on the fixed point method outlined in Section \ref{Sec:II}. All assumptions (except Assumption 5$'$) are quite easy to check if $\omega(x)$ is given.

We show by a numerical example that both iterative schemes converge to the same solution. Moreover, from numerics it follows that this solution is unique and stable. Indeed, the maximal and minimal fixed points turn out to be equal for any trials and choice of parameters. Thus, by Corollary \ref{corollary:1}, the fixed point is unique. Moreover, the constructed fixed point solution is stable since it is located between stable solutions of the $f_0$- and $f_\tau$-field, \cite{KA}.
Notice that we have not given a mathematical verification of these observations.

Notice also that we have looked for the bump solutions under the assumption  $\Delta_\tau<\Delta_0$ and \eqref{eq:u-spec}. Thus, even if the constructed solution is unique, it does not necessarily mean that there are no other stable or unstable solution. However, the same type of reasoning as we performed here are no longer valid if we relax on these assumptions. Therefore we leave this problem for a future study.

%%%%%%%%%%%%%%%%%%%%%%%

\section{Acknowledgements}

The authors would like to thank
Professor Stephen Coombes (School of Mathematical Sciences,
University of Nottingham, United Kingdom), and  Professor Vadim Kostrykin (Johannes Gutenberg-University, Mainz, Germany) for many fruitful and
stimulating discussions during the preparation of this paper. John Wyller and Anna Oleynik
also wish to thank the School of Mathematical Sciences,
University of Nottingham for the kind hospitality during the stay.
This research was supported by  the Norwegian University of Life
Sciences. The work has also been supported by The Research Council
of Norway under the grant No.~ 178892 (eNEURO-multilevel modeling
and simulation of the nervous system) and the grant  No.~ 178901
(Bridging the gap: disclosure, understanding and exploitation of the
genotype-phenotype map).

\end{document}